\documentclass[12pt,reqno]{amsart}

\makeatletter
\def\l@paragraph{\@tocline{4}{0pt}{1pc}{7pc}{}}
\def\l@subparagraph{\@tocline{5}{0pt}{1pc}{7pc}{}}
\makeatother

\usepackage{blindtext}
\numberwithin{equation}{section}
\usepackage[margin=2.5cm]{geometry}
\usepackage{amsmath}
\usepackage{accents}
\usepackage{multirow}
\usepackage{placeins}
\usepackage{amsthm}
\usepackage{calrsfs}
\usepackage{booktabs,caption}
\usepackage[flushleft]{threeparttable}
\DeclareMathAlphabet{\pazocal}{OMS}{zplm}{m}{n}

\usepackage[super]{nth}
\usepackage{tikz-cd}

\usepackage{setspace}
\usepackage{multirow}
\usepackage{alltt}
\usepackage{amsfonts}
\usepackage{amssymb}
\usepackage{booktabs}
\usepackage{caption}

\usepackage{enumitem}
\usepackage{fancyhdr}
\usepackage{graphicx}
\usepackage{mathdots}
\usepackage{mathtools}
\usepackage{microtype}
\usepackage{multirow}
\usepackage{pdflscape}
\usepackage{pgfplots}
\usepackage{siunitx}
\usepackage{subcaption}
\usepackage{tabularx}
\usepackage{tikz}
\usetikzlibrary{arrows,decorations.markings}
\usepackage{tkz-euclide}
\usepackage[normalem]{ulem}
\usepackage[all]{xy}
\usepackage{imakeidx}
\usepackage[utf8]{inputenc}
\usepackage[T1]{fontenc}
\usepackage{color}
\usepackage{algorithm}
\usepackage{algorithmic}
\usepackage{bbm}
\usepackage{xcolor}
\usepackage{mathtools}
\usepackage{framed}
\usepackage{ulem}
\usepackage[style=3]{mdframed}

\newtheorem{thm}{Theorem}[section]

\newtheorem{prop}[thm]{Proposition}
\newtheorem{cor}[thm]{Corollary}
\newtheorem{lemma}[thm]{Lemma}

\theoremstyle{definition}
\newtheorem{defi}[thm]{Definition}

\theoremstyle{remark}
\newtheorem{rem}[thm]{Remark}

\newcommand{\abs}[1]{\left\lvert#1\right\rvert}

\newcommand{\bra}[1]{\left\langle#1\right\rangle}
\newcommand{\rbra}[1]{\left(#1\right)}
\newcommand{\sbra}[1]{\left[#1\right]}
\newcommand{\bbra}[1]{\left\{#1\right\}}

\newcommand{\N}{\mathbb{N}}
\newcommand{\Z}{\mathbb{Z}}
\newcommand{\Q}{\mathbb{Q}}

\newcommand{\R}{\mathbb{R}}
\newcommand{\C}{\mathbb{C}}

\newcommand{\RNum}[1]{\uppercase\expandafter{\romannumeral #1\relax}}

\DeclareMathOperator{\im}{Im}
\DeclareMathOperator{\re}{Re}

\DeclareMathOperator{\GL}{GL}
\DeclareMathOperator{\SL}{SL}

\DeclareMathOperator{\OO}{O}

\DeclareMathOperator{\Det}{Det}
\DeclareMathOperator{\Hom}{Hom}
\DeclareMathOperator{\Ind}{Ind}
\DeclareMathOperator{\ord}{ord}

\newlength{\dhatheight}

\makeatletter
\newcommand{\hathat}[1]{%
\begingroup%
  \let\macc@kerna\z@%
  \let\macc@kernb\z@%
  \let\macc@nucleus\@empty%
  \hat{\raisebox{.2ex}{\vphantom{\ensuremath{#1}}}\smash{\hat{#1}}}%
\endgroup%
}
\makeatother

\definecolor{xdxdff}{rgb}{0.49019607843137253,0.49019607843137253,1.}
\definecolor{uuuuuu}{rgb}{0.26666666666666666,0.26666666666666666,0.26666666666666666}

\definecolor{morange}{rgb}{1, 0.5, 0}

\usepackage[colorlinks=true, a4paper=true, pdfstartview=FitV,
linkcolor=blue, citecolor=red, urlcolor=blue]{hyperref}

\usepackage{lineno}
\usepackage{biblatex} 
\title[Equidist of integers via Std. Quad. Form under arithmetic constraints]{Equidistribution of integers represented by standard quadratic form under arithmetic constraints}
\thanks{The author is supported by a CSC grant from the Chinese government.}
\author{Yefei Ma}
\address{IMAG – UMR 5149, Université de Montpellier, France}
\email{yefei.ma@umontpellier.fr}
\date{\today}
\begin{document}
\begin{abstract}
We study the equidistribution of integers of the form $n= x_1^2 + \cdots + x_d^2$ under the arithmetic constraints given by $(\Z/p\Z)^d$. The first step in addressing this problem is to construct modular forms whose Fourier expansion coefficients correspond to the counting problem over the quadric in $\Z^d$ induced by the standard quadratic form, subject to the aforementioned arithmetic constraints. The weak modular property of these modular forms allows us to use representation theory to identify the congruence subgroup to which our modular forms correspond. We then establish a necessary and sufficient condition for functions on $(\Z/p\Z)^d$ that defines a cusp form. Finally, we conclude that the equidistribution phenomenon occurs locally on $p+1$ orbits for $d \geq 4$.
\end{abstract}
\maketitle
\section{Introduction}
\subsection{Description of the problem and statement of the results}
Throughout the article, we write \(\mathbb N^*:=\{1,2,3,\ldots\}\) for the set of positive integers.

Let the dimension $d \in \N^*$. We equip the vector space $\R^d$ with the standard symmetric bilinear form $Q: \R^d \times \R^d \to \R$ defined as $Q(x,y) = x_1y_1 + \cdots + x_dy_d$. For any prime $p$ and $n \in \N$, we want to study the distribution of $\{ x \in \Z^d : Q(x,x) = n\}$ in $(\Z/p\Z)^d$ as $n \to \infty$. By an abuse of notation, we still denote the symmetric bilinear form on $(\Z/p\Z)^d$ by $Q$ as defined in the standard way. For any odd prime $p$ and $a \in \Z/p\Z$, let $X_{p,d}(a)$ be defined as the quadric
\begin{align}\label{cardinality xp}
    X_{p,d}(a) :=\bbra{ x \in (\Z/p\Z)^d : Q(x,x) = a} \subset (\Z/p\Z)^d.
\end{align}
In addition, for $a \in \Z/4\Z$ we define
\begin{align}
    X_{2,d}(a) := \bbra{ x \in (\Z/2\Z)^d : Q(x,x) = a} \subset (\Z/2\Z)^d.
\end{align}
\begin{rem}
For $p=2$, the standard symmetric bilinear form $Q$ in $(\Z/2\Z)^d$ gives $Q(x+2,x+2) = Q(x,x)+ 4Q(x,1)+4Q(1,1)$, so this allows us to define the map $(\Z/2\Z)^d \to \Z/4\Z$, $x \mapsto Q(x,x)$. Hence, $X_{2,d}(a)$ is well defined for $a \in \Z/4\Z$.
\end{rem}
For any $n \in \N$, let $X_d(n)$ be defined as the quadric
\begin{align}
    X_d(n) :=\bbra{ x \in \Z^d: Q(x,x) = n} \subset \Z^d,
\end{align}
whose cardinality, which we call the representation number by $Q$, is denoted by $r_d(n)$.

For a fixed prime $p$ and any $x$ defined on a subset of $\Z^d$, the map $x \mapsto x[p]$ defines the canonical projection of $\Z^d$ on $(\Z/p\Z)^d$.

We will prove the following result.
\begin{thm}\label{theorem equidistribution}
For $d \geq 4$ and $p$ an odd prime, let $a \in \Z/p\Z$. Then the integral points of $X_d(n)$ are asymptotically equidistributed on $X_{p,d}(n)$ in the following sense.
\begin{enumerate}
    \item If $a \neq 0$, then for $d \geq 5$ and any $f : X_{p,d}(a) \to \C$, the following holds
    \begin{equation}\label{eqn1 equi}
       \lim_{\substack{n \to \infty\\ n \equiv a [p]}} \frac{1}{\abs{X_d(n)}} \sum_{x \in X_d(n)}f(x[p]) = \frac{1}{\abs{X_{p,d}(a)}} \sum_{x \in X_{p,d}(a)} f(x).
    \end{equation}
    For $d=4$ and any $f : X_{p,4}(a) \to \C$, (\ref{eqn1 equi}) holds as $n \to \infty$ over odd integers $n \equiv a [p]$.

    \item If $a = 0$, then for $d \geq 5$ and any $f: X_{p,d}(0) \setminus \{0,\cdots,0\} \to \C$, the following holds
    \begin{equation}\label{eqn2 equi}
       \lim_{\substack{n \to \infty\\ n \in p\Z}} \frac{1}{\abs{\bbra{x \in \Z^d\setminus (p\Z)^d: Q(x,x)=n}}} \sum_{\substack{ x \in \Z^d \setminus (p\Z)^d,\\ Q(x,x)=n}}f(x[p])= \frac{1}{\abs{X_{p,d}(0)}-1} \sum_{\substack{x \in X_{p,d}(0),\\x \neq 0}} f(x).
    \end{equation}
    For $d=4$ and any $f : X_{p,4}(0) \setminus \{0,\cdots,0\} \to \C$, (\ref{eqn2 equi}) holds as $n \to \infty$ over odd integers $n \in p\Z$.

\end{enumerate}
\end{thm}
\begin{rem}
We can readily verify that if $a = 0$, the integral points of $\bbra{x \in (p\Z)^d: Q(x,x)=n}$ are asymptotically equidistributed on $\{0,\cdots,0\} \subset (\Z/p\Z)^d$ as $n \to \infty$ over the integers $n \in p\Z$. Explicitly, for $n \in p^2 \Z$ we have $\abs{\bbra{x \in (p\Z)^d: Q(x,x)=n}} = r_d(n/p^2)$; and for $n \in p\Z \setminus p^2 \Z$ we know $\bbra{x \in (p\Z)^d: Q(x,x)=n} = \varnothing$. Therefore, the equidistribution phenomenon on $\bbra{0,\cdots,0}$ actually occurs over the subsequence $n \in p^2 \Z$.
\end{rem}

The theorem says that for any odd prime $p$, the equidistribution phenomenon occurs locally on $p+1$ orbits. In the paper \cite{MR0931205}, Duke mentioned the usage of estimates on the Fourier coefficients of a holomorphic cusp form for the proof of the equidistribution phenomenon of lattice points on the 2-dimensional sphere $S^2$. We will adapt his method for the proof of our theorem.

We denote by $\pazocal{H}$ the upper-half plane in $\C$, that is
\begin{align*}
    \pazocal{H} = \bbra{\tau \in \C : \im(\tau)>0}.
\end{align*}
For any prime $p$, let $f:(\Z/p\Z)^d \to \C$ be a function on $(\Z/p\Z)^d$. We define the following \textit{weighted theta function} $\theta_f: \pazocal{H} \to \C$ to be
\begin{align}
    \mathcal{\theta}_f(\tau) &:= \sum_{n=0}^{\infty}\rbra{\sum_{x \in X_d(n)}f(x[p])}q^n,\quad q=e^{2\pi i \tau}, \tau \in \pazocal{H}.\label{theta fourier series}\\
&= \sum_{x \in \Z^d} f(x[p]) e^{2\pi i Q(x,x)\tau}, \quad \tau \in \pazocal{H}.\label{theta z^d}
\end{align}
We will show that this is a modular form for any prime $p$.
\begin{thm}\label{modular form}
Let $d \in \N^*$, $p$ be a prime and $f:(\Z/p\Z)^d\to \C$.
\begin{enumerate}
    \item If $p=2$, then the weighted theta function $\theta_f$ is a modular form of weight $\frac{d}{2}$ under $\Gamma_2$ where
\begin{align}\label{group gamma2}
    \Gamma_2 := \bbra{ g = \begin{bmatrix}
        a&b\\c&d
    \end{bmatrix} \in \SL_2(\Z): a,d=1[4],c=0[16]}.
\end{align}
\item If $p$ is an odd prime, then the weighted theta function $\theta_f$ is a modular form of weight $\frac{d}{2}$ under $\Gamma_p$ where
\begin{align}\label{group gamma_p}
    \Gamma_p := \bbra{ g = \begin{bmatrix}
        a&b\\c&d
    \end{bmatrix} \in \SL_2(\Z): a,d=1[4p],c=0[4p^2]}.
\end{align}
\end{enumerate}

\end{thm}
Next, we will find those functions $f$ defined on $(\Z/p\Z)^d$, which give rise to cusp forms $\theta_f$. The precise description is as follows.
\begin{thm}\label{cusp form}
Let $d \in \N^*$, $p$ be a prime and $f:(\Z/p\Z)^d\to \C$.
\begin{enumerate}
    \item If $p=2$, then the weighted theta function $\theta_f$ is a cusp form if and only if
\begin{equation}\label{vanishing condition p=2}
   \forall a \in \Z/4\Z: \sum_{x \in X_{2,d}(a)}f(x) =0,\quad f(1,\ldots,1)=0 \quad\text{ and }\quad f(0,\ldots,0)=0.
\end{equation}
    \item If $p$ is an odd prime, then the weighted theta function $\theta_f$ is a cusp form if and only if
\begin{equation}\label{vanishing condition}
   \forall a \in \Z/p\Z: \sum_{x \in X_{p,d}(a)}f(x) =0 \quad\text{ and }\quad f(0,\ldots,0)=0.
\end{equation}
\end{enumerate}
\end{thm}
\begin{rem}
For any odd prime $p$, by Witt theorem the condition (\ref{vanishing condition}) above could also be written as
\begin{equation}
    \sum_{g \in \OO(Q,(\Z/p\Z)^d)}gf=0,
\end{equation}
where $\OO(Q,(\Z/p\Z)^d)$ is the orthogonal group of $Q$ over $(\Z/p \Z)^d$ defined by $\OO(Q,(\Z/p\Z)^d)= \bbra{ g \in \GL(\Z/p\Z)^d : \forall u \in (\Z/p\Z)^d, Q(gu,gu)=Q(u,u) }$.
\end{rem}
We will prove Theorem \ref{modular form} and \ref{cusp form} in \S \ref{section 2} and Theorem \ref{theorem equidistribution} in \S \ref{section 3}.
\subsection*{Acknowledgements} I would like to thank Jean-François Quint for numerous inspiring discussions and for his help during the preparation of this manuscript. I also thank Guillaume Ricotta for helpful discussions. Finally, I thank the referee for their careful reading and valuable comments.
\section{Construction of cusp forms}\label{section 2}
In this section we prove Theorem \ref{modular form} and \ref{cusp form}. As for classical $\theta$-functions (cf. \cite{MR0344216}), the modularity property (which will be defined later in Definition \ref{defi weakly modular}) in Theorem \ref{modular form} relies on the use of a Poisson summation formula. To write this formula, we recall facts from abstract harmonic analysis in \S\ref{section poisson summation}. Next we introduce the language of group actions in \S\ref{section Cocycles and linear actions} so that the Poisson summation formula and the construction of modular forms can be interpreted in this way. We then recall the definition of modular forms and their properties in \S\ref{section Generalities on modular forms}, where we discuss the integral and half-integral weight modular forms separately. Our detailed construction and proofs take place in \S\ref{subsection Weak modularity and equivariance} and \S\ref{subsection proof of two theorems}.
\subsection{General Poisson summation formula}\label{section poisson summation}
Let $A$ be a locally compact Abelian group with Haar measure $dx$ and Pontryagin dual $\widehat{A}=\Hom_{\text{cont}}(A,S^1)$. For $\phi \in L^1(A)$, define its Fourier transform  to be the map $\mathcal{F}_A(\phi):\widehat{A} \to \C$ (or simply $\mathcal{F}(\phi)$ when there is no possible confusion) given by
\begin{equation}
    \mathcal{F}_A(\phi)(\chi) := \int_A \phi(x) \overline{\chi(x)} dx, \quad \chi \in \widehat{A}.
\end{equation}
Choosing the Plancherel measure $d\chi$ on $\widehat{A}$ gives rise to a canonical unitary equivalence $L^2(A) \cong L^2(\widehat{A})$ and the inversion formula $\mathcal{F}_{\widehat{A}}\mathcal{F}_{A}(\phi)(x) = \phi(-x)$ for all $x \in A$, where we have considered a function on $\hat{\hat{A}}$ as a function on $A$.

If $H \subset A$ is closed, denote its annihilator $ H^{\perp} := \bbra{ \chi \in \widehat{A}: \chi(x) = 1 \quad \forall x \in H}$. The \textit{general Poisson summation formula} (cf. \cite[Theorem 3.6.3]{MR2457798}) states that for $\phi \in L^1(A)$, if $\mathcal{F}_A(\phi)\rvert_{H^{\perp}} \in L^1(H^{\perp})$, then
\begin{equation}\label{poisson summation}
    \int_H \phi(xh)dh = \int_{H^{\perp}} \mathcal{F}_A(\phi)(\chi)\chi(x)d\chi,
\end{equation}
for all $x \in A$, where the Haar measure in $H^{\perp} \cong \widehat{A/H}$ is the Plancherel measure with respect to the Haar measure chosen in $A/H$.

In all our applications, the group $A$ will be of the form $A= (\Z/N\Z)^{d_1} \times \R^{d_2}$ for $N \in \N_+$, $d_1,d_2 \in \N$. Recall that we have defined the standard symmetric bilinear form on $\R^{d_2}$ as in the beginning of \S1, and for any $N \in \N_+$, we still denote the symmetric bilinear form on $(\Z/N\Z)^{d_1}$ by $Q$ as defined in the standard way. We then identify $A$ with its dual group $\widehat{A}$ by pairing
\begin{align*}
    A &\cong \widehat{A}\\
   (s,t)&\mapsto \chi_{s,t}(l,\xi) = e^{2\pi i \rbra{\frac{Q(s,l)}{N}+Q(t,\xi)}},
\end{align*}
where $s,l \in (\Z/N\Z)^{d_1}$ and $t,\xi \in \R^{d_2}$. In addition, we equip such a group $A$ with the Haar measure defined as the product of the counting measure on $(\Z/N\Z)^{d_1}$ and the standard Lebesgue measure on $\R^{d_2}$. By these choices, we will thus consider the Fourier transform of a function on $A$ as a function on $A$. Note that this means that the Plancherel measure associated with the counting measure differs from the standard one by a normalizing constant $\frac{1}{N^{d_1}}$ which gives the Fourier inversion formula $\mathcal{F}_{\widehat{A}}\mathcal{F}_{A}(\phi)(x) = \frac{1}{N^{d_1}}\phi(-x)$ for all $x=(s,t) \in A$ and $\phi \in L^1(A) \cap L^2(A)$. In addition, one can know from \cite{MR1970295} that the Schwartz space $S_e(U)$ is defined as the space of rapidly decreasing functions on some Eucliean space $U \subset \R^d$ so that the Fourier transform is an automorphism on this space. As a generalization of the classical Schwartz space, we define the \textit{Schwartz space} $S(A)$ to be
\begin{align*}
    S(A) := \bbra{\phi \in L^1(A) \cap L^2(A) : \forall s \in (\Z/N\Z)^{d_1}, (t \mapsto \phi(s,t)) \in S_e(\R^{d_2})}.
\end{align*}
We will always take $\phi$ from $S(A)$ with the Fourier transform defined as
\begin{equation}\label{convention}
    \mathcal{F}_A(\phi)(l,\xi) = \sum_{s \in (\Z/N\Z)^{d_1}}\int_{\R^{d_2}}\phi(s,t)e^{-2\pi i \rbra{\frac{Q(s,l)}{N}+Q(\xi,t)}}dt.
\end{equation}
In particular, we will use the following. 

\begin{prop}[The Poisson summation formula for $(\Z/N\Z)^d\times \R^d$]\label{corollary poisson}
For any $N \geq 1$ and $\phi \in S((\Z/N\Z)^d \times \R^d)$, we have
\begin{align}\label{poisson summation for lattice}
     \sum_{t \in \Z^d} \phi(t[N],t) = \frac{1}{N^d} \sum_{\xi \in \Z^d} \mathcal{F}(\phi)\rbra{-\xi[N],N^{-1}\xi}.
\end{align}
\end{prop}
\begin{proof}
For the group $A = ((\Z/N\Z)^d\times \R^d,+)$, we identify it with its dual group $\widehat{A}$ via the pairing $(s,t) \mapsto \chi_{(s,t)}$ where $\chi_{s,t}(l,\xi) = e^{2\pi i \rbra{\frac{Q(s,l)}{N}+Q(t,\xi)}}$ for $s,l \in (\Z/N\Z)^d$ and $t,\xi \in \R^d$. If we take $H = \bbra{(t[N],t): t\in \Z^d }$ as the closed subgroup inside $A$, then to identify $H^{\perp}$ we need to find those $(l,\xi) \in \widehat{A} \cong (\Z/N\Z)^d\times \R^d$ such that for any $(t[N],t) \in H$, we have $e^{2\pi i \rbra{\frac{Q(t,l)}{N}+Q(t,\xi)}} = 1$. In other words, we need $\frac{l}{N}+\xi \in \Z^d$, which means for $l \in (\Z/N\Z)^d$, we have $\xi = -\frac{l}{N}+ n$ for some $n \in \Z^d$. Equivalently speaking, we can write $H^{\perp}$ as
\begin{align*}
    \quad H^{\perp}= \bbra{ (-\xi[N],N^{-1}\xi): \xi \in \Z^d }.
\end{align*}
 For $\phi \in S((\Z/N\Z)^d \times \R^d)$, the Fourier transform formula is given by
\begin{equation}\label{fourier r,zp}
    \mathcal{F}(\phi)(l,\xi) = \sum_{s \in (\Z/N\Z)^d}\int_{\R^d}\phi(s,t)e^{-2\pi i \rbra{\frac{Q(s,l)}{N}+Q(\xi,t)}}dt.
\end{equation}
By the choice of the counting measure on $(\Z/N\Z)^d$, the normalizing factor $\frac{1}{N^d}$ for the Fourier inversion formula gives rise to the normalizing factor for the Poisson summation formula and the result follows by setting $x = 0$ in the sense of equation (\ref{poisson summation}).
\end{proof}
\subsection{Cocycles and linear actions}\label{section Cocycles and linear actions}
In this subsection, we will recall how the data of a group action and a cocycle produce a linear representation of the group. More precisely, given a group $G$, acting on some set $X$, and given another group $H$, we say that a \textit{cocycle} is a map $\varphi: G \times X \to H$ such that for any $g_1,g_2 \in G$ and any $j \in X$, the following holds
\begin{equation}\label{cocycle property}
    \varphi(g_1g_2,j) = \varphi(g_1,g_2 \cdot j)\varphi(g_2, j).
\end{equation}
 Let $V$ be a vector space. Denote by $V^X$ the space of $V$-valued functions on $X$. We have the following statement.
\begin{prop}\label{prop cocycle defines a representation}
 Let $G,H$ be two groups and let $G$ act on some set $X$. Given a cocycle $\varphi: G \times X \to H$ and a linear representation $\pi: H \to \GL(V)$, then for any $g \in G$, $F \in V^X$ and $j \in X$, the formula
\begin{equation}\label{formula cocycle action}
(g \cdot F)(j) := \pi \circ \varphi(g,g^{-1}\cdot j) F(g^{-1} \cdot j)
\end{equation}
defines an action of $G$ on $V^X$, which gives rise to a linear representation $\Pi: G \to \GL(V^X)$.
\end{prop}
\begin{proof}
    For $g_1,g_2 \in G$ the two quantities
\begin{align*}
    (g_1 \cdot (g_2 \cdot F))(j) &= \pi\circ\varphi(g_1,g_1^{-1}\cdot j)(g_2 \cdot F)(g_1^{-1}\cdot j)\\
    &= \pi \circ \varphi(g_1,g_1^{-1}\cdot j)\pi \circ \varphi(g_2,g_2^{-1}\cdot (g_1^{-1}\cdot j))F(g_2^{-1}\cdot (g_1^{-1}\cdot j)),\\
    ((g_1g_2)\cdot F)(j) &= \pi \circ \varphi(g_1g_2,(g_2^{-1}g_1^{-1})\cdot j)F((g_2^{-1}g_1^{-1})\cdot j)
\end{align*}
are equal by the defining cocycle property (\ref{cocycle property}). Thus it is indeed an action.
\end{proof}

We will now interpret modular forms in terms of group actions.
\subsection{Generalities on modular forms}\label{section Generalities on modular forms}
For \(R\in\{\mathbb Q,\mathbb R\}\), we denote by \(\GL_2^+(R)\) the subgroup
of \(\GL_2(R)\) consisting of matrices with positive determinant.
\subsubsection{The case of integral weight}
We recall standard definitions and properties of modular forms of weight $k \in \N$.

Let $\Omega(\pazocal{H})$ be the space of holomorphic functions from $\pazocal{H}$ to $\C$. For an element $g = \begin{bsmallmatrix}a&b\\c&d\end{bsmallmatrix} \in \GL_2^+(\R)$, we let it act on $\pazocal{H}$ in the standard way: that is $g \cdot \tau :=\frac{a\tau+b}{c\tau+d}$. For such $g$, we define the map $j: \GL_2^+(\R) \times \pazocal{H} \to \C$ such that $(g,\tau)\mapsto c\tau + d$. This map serves as our first example of a cocycle, that is for $g_1,g_2 \in \GL_2^+(\R)$ and $\tau \in \pazocal{H}$, we have $j(g_1g_2,\tau) = j(g_1,g_2\cdot\tau) j(g_2,\tau)$.

Let $g \in \GL_2^+(\R)$, for $k \in \N$ and $u \in \C$, we define the operator $[g]_{k,u}$ on functions in $\Omega(\pazocal{H})$ by
\begin{equation}\label{left action}
    ([g]_{k,u} \cdot \psi )(\tau):=j(g^{-1},\tau)^{-k} \Det(g)^{-u}\psi(g^{-1} \cdot \tau).
\end{equation}
\begin{prop}\label{prop left action}
For $k \in \N$ and $u \in \C$, the map $g \mapsto [g]_{k,u}$ defines a left action of $\GL_2^+(\R)$ on $\Omega(\pazocal{H})$.
\end{prop}
\begin{proof}
For $g \in \GL_2^+(\R)$, the map $j(g,\tau)$ is a holomorphic function of $\tau$, thus the action defined in Proposition \ref{prop cocycle defines a representation} preserves the space $\Omega(\pazocal{H})$, the conclusion follows from the cocycle property of $j(g,\tau)$.
\end{proof}
\begin{rem}
\begin{enumerate}
    \item In the literature, for $g \in \GL_2^+(\R)$ and $\psi \in \Omega(\pazocal{H})$, one may usually find the right action defined by
$\psi[g]_{k,u}(\tau):=j(g,\tau)^{-k}\Det(g)^u\psi(g\cdot \tau)$. For the purpose of the representation theory which we will develop later, we stick to the left action version.
\item If $g \in \SL_2(\R)$ and $\psi \in \Omega(\pazocal{H})$, we have $([g]_{k,u}\cdot \psi)(\tau) = j(g^{-1},\tau)^{-k}\psi(g^{-1} \cdot \tau)$, which is independent of $u$, so we simply denote $[g]_{k,u}$ by $[g]_k$ in this case. Note that if there is no ambiguity caused, we may omit the parameter $u$ when talking about the operator for any element of a subgroup of $\SL_2(\R)$.
\end{enumerate}
\end{rem}

We now restrict ourselves to \textit{congruence subgroups} of $\SL_2(\Z)$, i.e. subgroups $\Gamma$ containing $\Gamma(N)$ for some $N \geq 1$, where 
$$\Gamma(N):=\bbra{g \in \SL_2(\Z): g \equiv 1 [N] }.$$ 
Standard examples are \begin{align*}
\Gamma_1(N) := \bbra{g \in \SL_2(\Z): c \equiv 0 [N], a,d \equiv 1 [N] },\quad
\Gamma_0(N) := \bbra{g \in \SL_2(\Z): c \equiv 0 [N] },
\end{align*}
and every congruence subgroup has a finite index in $\SL_2(\Z)$.

A particular example which is interesting throughout our article is the congruence subgroup $\Gamma_1(4)$. We write $\alpha =\begin{bsmallmatrix}1&1\\0&1\end{bsmallmatrix}$ and $\gamma =\begin{bsmallmatrix}0&1\\-4&0\end{bsmallmatrix}$ which we consider as matrices in $\GL_2(\Z[\frac{1}{2}])$. Let $\Lambda$ be the subgroup of $\GL_2(\Z[\frac{1}{2}])$ generated by $\alpha,\gamma$. Moreover, let $\beta := \gamma \alpha \gamma^{-1} = \begin{bsmallmatrix}
    1&0\\-4&1
\end{bsmallmatrix}$ so that $\beta \in \SL_2(\Z)$. The following lemma lists some key properties of $\Gamma_1(4)$.

\begin{lemma}\label{lemma Lambda inside SL2Z}
The intersection of $\Lambda$ and $\SL_2(\Z)$ is $\Gamma_1(4)$, which is generated by $\alpha,\beta$.
\end{lemma}
\begin{proof}
We first claim that $\Lambda \cap \SL_2(\Z)$ is spanned by $\alpha,\beta$, i.e. $\Lambda \cap \SL_2(\Z) = \bra{\alpha,\beta}$. The inclusion $\langle \alpha,\beta\rangle \subseteq \Lambda \cap \SL_2(\mathbb Z)$ is immediate since \(\alpha,\beta\in \Lambda\cap \SL_2(\mathbb Z)\). Conversely, let $g \in \Lambda \cap \SL_2(\Z)$. Since \(\Lambda\) is generated by
\(\alpha\) and \(\gamma\), we may write $g =\alpha^{m_0}\gamma^{n_1}\cdots\gamma^{n_r}\alpha^{m_r}$ with $m_i,n_i \in \Z$. Since \(\det(\alpha)=1\) and
\(\det(\gamma)=4\), the condition \(\det(g)=1\) gives $\sum_{i=1}^r n_i=0$. Write
\[
    n_i=2u_i+\varepsilon_i,
    \qquad u_i=\left\lfloor \frac{n_i}{2}\right\rfloor,
    \qquad \varepsilon_i\in\{0,1\}.
\]
Using \(\gamma^2=-4I\), we obtain $\gamma^{n_i}=(-4)^{u_i}\gamma^{\varepsilon_i}$. Hence,  $g$ can be rewritten as 
\begin{align*}
    g=(-4)^U
      \alpha^{m_0}\gamma^{\varepsilon_1}\alpha^{m_1}\cdots
      \gamma^{\varepsilon_r}\alpha^{m_r},
    \qquad
    U:=\sum_{i=1}^r u_i .
\end{align*}
Since $0=\sum_{i=1}^r n_i
     =2U+\sum_{i=1}^r \varepsilon_i$, the number \(\sum_i \varepsilon_i\) of remaining occurrences of \(\gamma\) is even. Say $\sum_{i=1}^r \varepsilon_i=2k$. Then \(U=-k\), and therefore $(-4)^U=(-4)^{-k}=(\gamma^{-2})^k.$ After deleting the factors \(\gamma^0\) and collecting adjacent powers of
\(\alpha\), we may write
\[
g=(\gamma^{-2})^k
      \alpha^{a_0}\gamma\alpha^{a_1}\gamma\alpha^{a_2}
      \cdots
      \gamma\alpha^{a_{2k-1}}\gamma\alpha^{a_{2k}},
\]
for some \(a_j\in \mathbb Z\). We may move
one copy of \(\gamma^{-2}\) next to every second occurrence of \(\gamma\). Thus
\[
\begin{aligned}
    g
    &=
      \alpha^{a_0}
      \gamma\alpha^{a_1}\gamma\gamma^{-2}\alpha^{a_2}
      \cdots
      \gamma\alpha^{a_{2k-1}}\gamma\gamma^{-2}\alpha^{a_{2k}}  \\
    &=
      \alpha^{a_0}
      \gamma\alpha^{a_1}\gamma^{-1}\alpha^{a_2}
      \cdots
      \gamma\alpha^{a_{2k-1}}\gamma^{-1}\alpha^{a_{2k}}.
\end{aligned}
\]
But $
    \gamma\alpha^{a_j}\gamma^{-1}
    =(\gamma\alpha\gamma^{-1})^{a_j}
    =\beta^{a_j}.$ Therefore, 
    \[
    g=
    \alpha^{a_0}\beta^{a_1}\alpha^{a_2}\cdots
    \beta^{a_{2k-1}}\alpha^{a_{2k}}
    \in \langle \alpha,\beta\rangle .
\]
This proves $ \Lambda \cap \SL_2(\mathbb Z)=\langle \alpha,\beta\rangle .$

Next, we want to prove that $\bra{\alpha,\beta} = \Gamma_1(4)$. For any $g \in \bra{\alpha,\beta}$, it follows immediately that $g \in \Gamma_1(4)$ because both of its generators $\alpha,\beta \in \Gamma_1(4)$. Conversely, we prove by induction on $k$ that the statement $S(k)$ as follows is true for all $k \in \N$,
\begin{align*}
    S(k): \text{ For all }g \in \Gamma_1(4), \text{ if }g = \begin{bsmallmatrix}
    a&b\\4c&d
\end{bsmallmatrix} \text{ with }\abs{c} \leq k, \text{ then }g \in \bra{\alpha,\beta}.
\end{align*}

For $k=0$, we know $g \in \Gamma_1(4)$ implies that $ad = 1$. Since $a,d = 1[4]$, we have $a = d = 1$, which means $g = \begin{bsmallmatrix}
    1&b\\0&1
\end{bsmallmatrix} = \alpha^b \in \bra{\alpha,\beta}$. Hence, $S(0)$ is true. Next we suppose $S(k)$ is true for some $k \in \N$, and we prove that $S(k+1)$ is true. We take $g = \begin{bsmallmatrix}
    a&b\\4c&d
\end{bsmallmatrix} \in \Gamma_1(4)$ such that $\abs{c} \leq k+1$. If $\abs{c} \leq k$, then by induction hypothesis $g \in \bra{\alpha,\beta}$. For $\abs{c}=k+1$, one can find $m,n \in \Z$ such that
\begin{align*}
    g' &= g \alpha^m = \begin{bmatrix}
       a&b+ma\\c&c+md
    \end{bmatrix} = \begin{bmatrix}
    a'&b'\\ 4c'&d'
\end{bmatrix} \quad \text{with } c'= c \text{ and }\abs{d'} < 2 \abs{c'},\\
g'' &= g'\beta^n = \begin{bmatrix}
    a+nb&b\\c+nd&d
\end{bmatrix} =\begin{bmatrix}
    a''&b''\\ 4c''&d''
\end{bmatrix} \quad \text{with } d''= d' \text{ and }4\abs{c''} < 2 \abs{d''}.
\end{align*}
Then it follows that $4\abs{c''} < 2 \abs{d''} < 4 \abs{c}$, which means $g'' \in \bra{\alpha,\beta}$ by induction hypothesis. In this case, we note that $g= g'' \beta^{-n} \alpha^{-m} \in \bra{\alpha,\beta}$, so we have proved that $S(k+1)$ is true. This means for any $g \in \Gamma_1(4)$, we have $g \in \bra{\alpha,\beta}$. In summary, we have proved the statement.

\end{proof}
\begin{rem}
The above generating set is consistent with the standard computational
description of \(\Gamma_1(4)\). Indeed, algorithms such as the Todd--Coxeter
algorithm or the Farey-symbol method can be used to compute generators of
congruence subgroups \(\Gamma_1(N)\), and are implemented for instance in
SageMath.\footnote{\url{https://www.sagemath.org/}} For \(N=4\), SageMath's
default Farey-symbol algorithm returns $A=
    \begin{bsmallmatrix}
        1&1\\0&1
    \end{bsmallmatrix}$, $B=
    \begin{bsmallmatrix}
        -3&1\\-4&1
    \end{bsmallmatrix}$. Since $A^{-1}B=
    \begin{bsmallmatrix}
        1&0\\-4&1
    \end{bsmallmatrix}
    =\beta$,
the pair \((A,A^{-1}B)\) is precisely \((\alpha,\beta)\).
\end{rem}
We next define weak modularity (cf. \cite[Chapter VII, \S 2]{MR0344216} for $\Gamma = \SL_2(\Z)$).

\begin{defi}\label{defi weakly modular}
    Let $\psi \in \Omega(\pazocal{H})$ and $k \in \N$ and let $\Gamma \subset \SL_2(\Z)$ be a congruence subgroup. We say $\psi$ is \textit{weakly modular of weight $k$ under $\Gamma$} if the set $\bbra{ g \in \SL_2(\R) : [g]_k \cdot \psi=\psi}$ contains $\Gamma$.
\end{defi}
\begin{rem}
If $k$ is odd and $-I \in \Gamma$, then $[-I]_k \cdot \psi = - \psi$, so no nonzero can be  weakly modular. This motivates our latter restriction to subgroups of $\Gamma_1(4)$.
\end{rem}

The next lemma prepare the definition of holomorphicity at cusps for weakly modular $\psi \in \Omega(\pazocal{H})$.

\begin{lemma}\label{lemma weakly modular}
Let $h \in \GL_2^+(\Q)$ and $\Gamma \subset \SL_2(\Z)$ be a congruence subgroup, then $h\Gamma h^{-1} \cap \SL_2(\Z)$ is also a congruence subgroup. Moreover, if $\psi \in \Omega(\pazocal{H})$ is weakly modular of weight $k$ under $\Gamma$, then $[h]_{k,u} \cdot \psi$ is weakly modular of weight $k$ under $h\Gamma h^{-1} \cap \SL_2(\Z)$.
\end{lemma}
\begin{proof}
Both claims follow from the definition of congruence subgroups and the operation $[hgh^{-1}]_{k,u} \cdot ([h]_{k,u}  \cdot \psi) = [h]_{k,u}\cdot \psi $.
\end{proof}

Given a field \(K\), we always identify the set \(\mathbb{P}^1_K\) of
lines in \(K^2\) with \(K \cup \{\infty\}\) in the usual way: we identify
any \(x \in K\) with the line \([x:1]:=K(x,1)\), and \(\infty\) with the line
\([1:0]:=K(1,0)\). Under this identification, the natural action of
\(\GL_2(K)\) may be written as follows: for $g = \begin{bsmallmatrix}
    a&b\\c&d
\end{bsmallmatrix} \in \GL_2(K)$
\begin{align*}
    g \cdot x = \frac{ax+b}{cx+d} \quad \text{for } x \in K, &\quad\text{if }cx + d \neq 0\\
    g \cdot \rbra{-\frac{d}{c}} = \infty, \quad g \cdot \infty = \frac{a}{c}, &\quad\text{if } c \neq 0\\
    g \cdot \infty = \infty, &\quad\text{if } c = 0.
\end{align*}
We will consider this action with $K = \R$ now and $K = \mathbb{F}_p$ with $p$ prime, later. Note that $\SL_2(\Z)$ acts transitively on $\mathbb{P}^1_{\Q} \cong \Q \cup \{\infty\}$, so any congruence subgroup of $\SL_2(\Z)$ has finitely many orbits in $\mathbb{P}^1_{\Q}$ due to the finite index property of congruence subgroups. Let $\Gamma$ be a congruence subgroup, then a $\Gamma$-orbit in $\mathbb{P}^1_{\Q}$ is called a \textit{cusp} of $\Gamma$.

\begin{defi}\label{cusp definition}
  Let $k\in\N$ and $\Gamma \subset \SL_2(\Z)$ be a congruence subgroup and $\psi \in \Omega(\pazocal{H})$ be weakly modular of weight $k$ under $\Gamma$.
  \begin{enumerate}
      \item We say $\psi$ is \textit{holomorphic at $\infty$} if it has a Fourier expansion $\psi(\tau)=\psi'(q_m)=\sum_{n=0}^{\infty} a_nq_m^n$ where $q_m = e^{2\pi i \tau/m}$ for some $m \in \Z^+$ or equivalently $\lim_{\im(\tau)\to \infty}\psi(\tau)$ exists in $\C$. We say this $\psi$ is \textit{zero at $\infty$} when $a_0=0$ or equivalently $\lim_{\im(\tau)\to \infty}\psi(\tau)=0$.
      \item We say $\psi$ is \textit{holomorphic at $s \in \Q$} if there exists some $h \in \GL_2^+(\Q)$ with $h \cdot s = \infty$ and $u \in \C$ such that the function $[h]_{k,u} \cdot \psi$ is holomorphic at $\infty$. We say this $\psi$ is \textit{zero at $s$} if the function $[h]_{k,u} \cdot \psi$ is zero at $\infty$.
      \item We say $\psi$ is \textit{holomorphic at all the cusps} if $\psi$ is holomorphic at all $s \in \mathbb{P}^1_{\Q}$. We say this $\psi$ is \textit{zero at all the cusps} if $\psi$ is zero at all $s \in \mathbb{P}^1_{\Q}$.
  \end{enumerate}
\end{defi}
\begin{rem}\label{rem holomorphic aat cusps}

To define the holomorphicity condition at $\Q$ in a convenient way, we observe that for any $s \in \mathbb{P}^1_{\Q}$, one can always find some $h \in \GL_2^+(\Q)$ such that $h \cdot s = \infty$. This motivates us to define the holomorphicity condition of $\psi$ at $s\in \Q$ by sending it to $\infty$ via such $h$ and consider the holomorphicity condition of $[h]_{k,u} \cdot \psi$ at $\infty$ for some $u \in \C$. By Lemma \ref{lemma weakly modular}, we know that $[h]_{k,u} \cdot \psi$ is weakly modular, so this definition makes sense.
\end{rem}

We are now ready to give the full definition of modular and cusp forms of integral weight; however, since we later extend this to half-integral weights, we omit the separate integral case and refer to Definition \ref{defi half integral modular forms}, which covers both.

\subsubsection{The case of half-integral weight}\label{section 2.3.2}
We define modular forms of weight $k \in \frac{1}{2} \Z$ under congruence subgroups of $\Gamma_1(4)$ in this subsection.

Throughout this article, for $z \in \C$, the square root $\sqrt{z}$ or $z^{\frac{1}{2}}$ is taken in the principal branch, that is $\arg(\sqrt{z}) \in (-\pi/2,\pi/2]$. Moreover, for $m \in \Z$, by $z^{\frac{m}{2}}$ we mean $(\sqrt{z})^m$. We first define some notations.
\begin{defi}\label{defi extended symbol and varepsilon}
    Let $d$ be an odd integer and $c \in \Z$.
    \begin{enumerate}
        \item We write $\rbra{\frac{c}{d}}$ for the Jacobi symbol (Legendre if $d$ is prime), extended by $\rbra{\frac{0}{d}}=1$ if $d = \pm 1$, 0 otherwise and $\rbra{\frac{c}{d}} = \text{sgn}(c)\rbra{\frac{c}{\abs{d}}}$ if $d <0$. Note that this extension ensures that the usual formula $\rbra{\frac{-1}{d}} = (-1)^{(d-1)/2}$ holds for any odd $d$.
        
        \item Define $\varepsilon_d = \sqrt{\rbra{\frac{-1}{d}}}$, so $\varepsilon_d = 1$ if $d \equiv 1 [4]$ and $i$ if $d \equiv 3 [4]$.
    \end{enumerate}
\end{defi}

With these definitions, we define the map $h: \Gamma_0(4) \times \pazocal{H} \to \C$ by
\begin{align}\label{formula h}
    h(g,\tau) := \varepsilon^{-1}_d \rbra{\frac{c}{d}} \sqrt{c\tau + d} \quad \text{ for }g = \begin{bsmallmatrix}
      a&b\\c&d
    \end{bsmallmatrix} \in \Gamma_0(4), \tau \in \pazocal{H}.
\end{align}
\begin{rem}\label{rem choice of square roots}
    The map $h(g,\tau)$ is a choice of square roots of $\rbra{\frac{-1}{d}}(c\tau + d)$ by Definition \ref{defi extended symbol and varepsilon}.
\end{rem}
This map appears in the transformation formula for the classical theta function $\Theta(\tau)= \sum_{x \in \Z} e^{2\pi i \tau x^2}$ for $\tau \in \pazocal{H}$, which is $\Theta(g\cdot \tau)= h(g,\tau)\Theta(\tau) \text{ for }g \in \Gamma_0(4),\tau \in \pazocal{H}$ (cf. \cite[\S 10.5]{MR1474964}). Moreover, we know that $\Theta$ is non-vanishing on $\pazocal{H}$ (cf. \cite[\S 21.12]{MR1424469}) so that $h$ is a cocycle.

Recall that $\Lambda$ is the subgroup of $\GL_2(\Z[\frac{1}{2}])$ generated by $\alpha,\gamma$, and its intersection with $\SL_2(\Z)$ is $\Gamma_1(4)$. We aim to define another cocycle $\bar{h}: \Lambda \times \pazocal{H} \to \C$ which generalizes $h$.

To do this, recall that by Remark \ref{rem choice of square roots}, the definition of the map $h$ involves a choice of square roots of $\rbra{\frac{-1}{d}}(c\tau + d)$ where $\rbra{\frac{-1}{d}}$ can take $\pm 1$ depending on $d$. For any $g = \begin{bsmallmatrix}
    a&b\\c&d
\end{bsmallmatrix} \in \Gamma_1(4)$, in fact $\rbra{\frac{-1}{d}}$ takes the single value 1. This suggests to define a central extension by $\Z/2\Z$ of $\GL_2^+(\Q)$, which we denote by $\pazocal{G}$ and is called the \textit{metaplectic group} over $\GL_2^+(\Q)$. In general, for half-integral modular forms under $\Gamma_0(4)$, one can follow the construction of \cite[\S IV.1]{koblitz1993}, which uses a four-sheeted covering, i.e. a $\Z/2\Z \times \Z/2\Z$ central extension, of $\GL_2^+(\Q)$. An excellent exposition of this topic can also be found in \cite{MR1102679}. We now give the explicit definition of $\pazocal{G}$.

Recall for $g = \begin{bsmallmatrix}
    a&b\\c&d
\end{bsmallmatrix} \in \GL_2^+(\Q)$, the map $j:\GL_2^+(\R) \times \pazocal{H} \to \C$ is $(g,\tau) \mapsto c\tau + d$. We define $\pazocal{G}$ to be the set of all ordered pairs $\dot{g} := (g,\phi)$, where $g \in \GL_2^+(\Q)$ and $\phi:\pazocal{H} \to \C$ is a holomorphic function such that
\begin{align}\label{four-sheeted covering}
    \phi(\tau)^2 = j(g,\tau), \quad \text{for all }\tau \in \pazocal{H}.
\end{align}

\begin{rem}\label{rem product rule}
   We can define the product rule for two elements $(g_1,\phi_1),(g_2,\phi_2) \in \pazocal{G}$ as $(g_1,\phi_1)(g_2,\phi_2)= (g_1g_2,\phi_3)$, where $\phi_3(\tau) = \phi_1(g_1\cdot \tau)\phi_2(\tau)$ for $\tau \in \pazocal{H}$. One can check by direct computation that $\pazocal{G}$ is a group under this operation (cf. \cite[\S IV.1]{koblitz1993}), and that $(\dot{g} = (g,\phi), \tau) \mapsto \phi(\tau)$ is a cocycle.
\end{rem}

For $\dot{g} \in \pazocal{G}$, $k \in \frac{1}{2}\Z$ and $u \in \C$, we define an operator $[\dot{g}]_{k,u}$ on the functions in $\Omega(\pazocal{H})$ by
\begin{align}
([\dot{g}]_{k,u}\cdot \psi)(\tau) := \phi(\tau)^{-2k} \Det(g)^{-u}\psi(g^{-1}\cdot \tau).
\end{align}
Again, when $g \in \SL_2(\Q)$, we will simply use $[\dot{g}]_{k}$ instead of $[\dot{g}]_{k,u}$ by omitting the parameter $u$. The cocycle property of $\phi$ given by Remark \ref{rem product rule} immediately implies that $g \mapsto [\dot{g}]_{k,u}$ defines a left action of $\pazocal{G}$ on $\Omega(\pazocal{H})$.

In addition, this construction gives rise to a homomorphism $P: \pazocal{G} \to \GL_2^+(\Q)$, which projects onto the first coordinate of the pair: $(g,\phi) \mapsto g$. It follows that $\ker(P)\simeq \Z/2\Z$ and we have the short exact sequence $1 \to \Z/2\Z \to \pazocal{G} \xrightarrow{P} \GL_2^+(\Q) \to 1$.

For any subgroup $\Gamma \subseteq \Gamma_1(4)$ and $g \in \Gamma$, we define $\tilde{\phi}(\tau):=h(g,\tau)$ and
\begin{align*}
    \widetilde{\Gamma}:= \bbra{(g,\tilde{\phi}): g \in \Gamma} \subset \pazocal{G}.
\end{align*}

For the subgroup $\Gamma_1(4) \subset \Lambda$, the existence of a lifting $L: \Gamma_1(4) \to \pazocal{G}: g \mapsto (g,\tilde{\phi})$ of the projection $P: \pazocal{G} \to \Gamma_1(4):(g,\tilde{\phi})\mapsto g$ is guaranteed by (\ref{formula h}). Moreover, the functional equation for the classical theta function $\Theta$ says that
\begin{align}
    \Theta(\tau) =  \rbra{\frac{i}{2\tau}}^{1/2}\Theta\rbra{-\frac{1}{4\tau}},
\end{align}
cf. \cite[(3.4)]{koblitz1993}. This suggests the following result.

\begin{prop}\label{prop cocycle h bar}
    There exists a unique cocycle $\bar{h}: \Lambda \times \pazocal{H} \to \C$ such that for any $\tau \in \pazocal{H}$
    \begin{align*}
        \bar{h}(\alpha,\tau)&=1,\\
        \bar{h}(\gamma,\tau)&=\sqrt{-4\tau}.
    \end{align*}
    Moreover, we have $\bar{h}(g,\tau)^2 = j(g,\tau)$ and also $\bar{h}$ is holomorphic in $\tau$.
\end{prop}
\begin{proof}
Let $u$ be a solution of the equation $\rbra{\frac{i}{2}}^{\frac{1}{2}} = \rbra{\frac{1}{4}}^u$. We can verify that the map $\bar{h}$ satisfies
\begin{align*}
    \Theta(g \cdot \tau) = \bar{h}(g,\tau) \Det(g)^{-u} \Theta(\tau)
\end{align*}
for $g=\alpha,\gamma$ and any $\tau \in \pazocal{H}$ and in general for any $g \in \Lambda$ and $\tau \in \pazocal{H}$. Since the determinant function $\Det(g)$ is multiplicative, the claim that $\bar{h}$ is a cocycle follows from the previous argument that $h$ is a cocycle. The holomorphicity of $\bar{h}$ in $\tau$ follows from the fact that $\Theta$ is holomorphic on $\pazocal{H}$.
\end{proof}
\begin{rem}
    On $\Gamma_1(4) \times \pazocal{H}$ the map $\bar{h}$ coincides with $h$
\begin{align*}
  \bar{h}|_{\Gamma_1(4) \times \pazocal{H}}= h,
\end{align*}
so that $\bar{h}$ can be seen as a generalization of the map $h$ by including the generator $\gamma \in \GL_2(\Z[\frac{1}{2}])$.
\end{rem}

For $g \in \Lambda$, by an abuse of notation we denote again by $\tilde{\phi}(\tau) := \bar{h}(g,\tau)$ and we set
\begin{align*}
    \tilde{g}:= (g, \tilde{\phi}) \in \pazocal{G}.
\end{align*}
Let $\widetilde{\Lambda}$ be the subgroup of $\pazocal{G}$ generated by $\widetilde{\alpha},\widetilde{\gamma}$, so by construction, we have built an isomorphism $\Lambda \to \widetilde{\Lambda}$. For $g \in \Lambda$, $k \in \frac{1}{2}\Z$ and $u \in \C$, we define the operator $[\tilde{g}]_{k,u}$ on the functions in $\Omega(\pazocal{H})$ by
\begin{align}
     ([\tilde{g}]_{k,u} \cdot \psi )(\tau):=\bar{h}(g^{-1},\tau)^{-2k} \Det(g)^{-u}\psi(g^{-1} \cdot \tau).
\end{align}
Analogously to the proof of Proposition \ref{prop left action} with the fact by Proposition \ref{prop cocycle h bar} that for every fixed $g \in \Lambda$ the cocycle $\bar{h}(g,\tau)$ is a holomorphic function on $\pazocal{H}$ so by Lemma \ref{prop cocycle defines a representation} we conclude that
\begin{prop}\label{prop left action half integral}
     For $k \in \frac{1}{2}\Z$ and $u \in \C$, the map $g \mapsto [\tilde{g}]_{k,u}$ defines a left action of $\Lambda$ on $\Omega(\pazocal{H})$.
 \end{prop}
 Again, when $g \in \Gamma \subseteq \Gamma_1(4) \subset \SL_2(\R)$, we simply replace $[\tilde{g}]_{k,u}$ with $[\tilde{g}]_k$ omitting the parameter $u$, if there is no ambiguity caused.

 This notion enables us to define a function as weak modular under $\tilde{\Gamma}$ for a congruence subgroup $\Gamma \subseteq \Gamma_1(4)$.
 \begin{defi}\label{defi weakly modular half integral}
     Let $\psi \in \Omega(\pazocal{H})$ and $k \in \frac{1}{2}\Z$ and let $\Gamma \subseteq \Gamma_1(4)$ be a congruence subgroup. We say that $\psi$ is \textit{weakly modular of weight $k$ under $\tilde{\Gamma}$} if for some $u \in \C$ the set $\{ g \in \Lambda: [\tilde{g}]_{k,u} \cdot \psi = \psi \}$ contains $\Gamma$.
 \end{defi}
Following the discussion after Lemma \ref{lemma weakly modular}, the action of $\pazocal{G}$ on $\mathbb{P}^1_{\Q}$ is defined to be the action of the image of $\pazocal{G}$ via the projection map $P: \pazocal{G} \to \GL_2^+(\Q)$ on $\mathbb{P}^1_{\Q}$.

Now we will explain what it means for $\psi \in \Omega(\pazocal{H})$ to be holomorphic at cusps, under our newly defined notion of weak modular.
\begin{defi}\label{defi holomorphic at cusps half integral}
   Let $k \in \frac{1}{2}\Z$ and $\Gamma \subseteq \Gamma_1(4)$ be a congruence subgroup and $\psi \in \Omega(\pazocal{H})$ be weakly modular of weight $k$ under $\tilde{\Gamma}$.
   \begin{enumerate}
        \item We say $\psi$ is \textit{holomorphic at $\infty$} if it has a Fourier expansion $\psi(\tau) = \psi'(q_m)=\sum_{n=0}^\infty a_n q_m^n$ where $q_m = e^{2\pi i \tau/m}$ for some $m \in \N^+$ or equivalently $\lim_{\im(\tau) \to \infty} \psi(\tau)$ exists in $\C$. We say this $\psi$ is zero at $\infty$ when $a_0=0$ or equivalently $\lim_{\im(\tau) \to \infty} \psi(\tau)=0$.
        \item We say $\psi$ is \textit{holomorphic at $s \in \Q$} if there exists some $\dot{h} \in \pazocal{G}$ with $\dot{h} \cdot s = \infty$ and $u \in \C$ such that the function $[\dot{h}]_{k,u} \cdot \psi$ is holomorphic at $\infty$. We say this $\psi$ is \textit{zero at $s \in \Q$} if the function $[\dot{h}]_{k,u} \cdot \psi$ is zero at $\infty$.
        \item We say $\psi$ is \textit{holomorphic at all the cusps} if $\psi$ is holomorphic at all $s \in \mathbb{P}^1_{\Q}$. We say this $\psi$ is \textit{zero at all the cusps} if $\psi$ is zero at all $s \in \mathbb{P}^1_{\Q}$.
    \end{enumerate}
\end{defi}
\begin{rem}

For $s \in \Q$, we can again find some $\dot{h}=(h,\phi(\tau))$ such that $\dot{h} \cdot s = \infty$, then by an analogous argument to Lemma \ref{lemma weakly modular}, we can find some $m \in \N^+$ with $4|m$ such that $[\dot{h}]_{k,u} \cdot \psi$ is invariant under $[(\begin{bsmallmatrix}
            1&m\\0&1
        \end{bsmallmatrix},1)]_k$, where $(\begin{bsmallmatrix}
            1&m\\0&1
        \end{bsmallmatrix},1) \in \widetilde{\Lambda}$. In other words,
        \begin{align*}
            ([\dot{h}]_{k,u} \cdot \psi )(\tau) = [(\begin{bsmallmatrix}
            1&m\\0&1
        \end{bsmallmatrix},1)]_k \cdot ([\dot{h}]_{k,u} \cdot \psi )(\tau) =([\dot{h}]_{k,u} \cdot \psi )(\tau-m).
        \end{align*}
        Thus our definition follows naturally.

\end{rem}

In Definition \ref{defi holomorphic at cusps half integral}(2), if we can find some $\dot{h_1} \in \pazocal{G}$ with $\dot{h_1} \cdot s = \infty$ and $u_1\in \C$ such that the function $[\dot{h_1}]_{k,u_1} \cdot \psi$ is zero at $\infty$, this automatically implies the same property of $[\dot{h_2}]_{k,u_2} \cdot \psi$ for any $\dot{h_2} \in \pazocal{G}$ with $\dot{h_2} \cdot s = \infty$ and $u_2 \in \C$.
\begin{lemma}\label{lemma holomorphic at s existence}
Let $k\in \frac{1}{2}\Z$ and $\Gamma \subseteq \Gamma_1(4)$ be a congruence subgroup and $\psi \in \Omega(\pazocal{H})$ be weakly modular of weight $k$ under $\tilde{\Gamma}$. For $s \in \Q$, if there exists some $\dot{h_1} \in \pazocal{G}$ with $\dot{h_1} \cdot s=\infty$ and $u_1 \in \C$ such that the function $[\dot{h_1}]_{k,u_1} \cdot \psi$ is zero at $\infty$, then the function $[\dot{h_2}]_{k,u_2} \cdot \psi$ is zero at $\infty$ for any $\dot{h_2} \in \pazocal{G}$ with $\dot{h_2} \cdot s = \infty$ and $u_2 \in \C$.
\end{lemma}
\begin{proof}
Note that if we have $\dot{h_1} \cdot s = \infty$, then for any $\dot{h_2} \in \pazocal{G}$ with $\dot{h_2} \cdot s=\infty$ we can deduce that $h_2h_1^{-1} \in \GL_2^+(\Q)_{\infty}$ the stabilizer of $\GL_2^+(\Q)$ at $\infty$. It can be characterized as follows
\begin{align*}
    \GL_2^+(\Q)_{\infty}= \bbra{\begin{bmatrix}
        a&b\\0&d
    \end{bmatrix}:a,b,d \in \Q, ad>0}.
\end{align*}
Since the function $[\dot{h_1}]_{k,u_1} \cdot \psi$ is zero at $\infty$, there exists some $m \in \N^*$ such that $[\dot{h_1}]_{k,u_1} \cdot \psi(\tau) = \sum_{n=0}^\infty a_n q_m^n$ with $q_m = e^{2\pi i \tau/m}$ and $a_0=0$, and it follows that for $\dot{h_2} \in \pazocal{G}$ we have
\begin{align*}
    \rbra{[\dot{h_2}]_{k,u_1} \cdot \psi }(\tau) = \rbra{[\dot{h_2}\dot{h_1^{-1}}]_{k,u_1}[\dot{h_1}]_{k,u_1}\cdot \psi }(\tau) = \varepsilon^{-2k} d^{-2k}a^{-u_1} \sum_{n=0}^\infty a_n e^{2\pi i (a\tau +b)n/dm}
\end{align*}
for some $h_1h_2^{-1} = \begin{bsmallmatrix}
    a&b\\0&d
\end{bsmallmatrix} \in \GL_2^+(\Q)_{\infty}$ and $\varepsilon \in \bbra{\pm 1}$. For any $u_2 \in \C$, the function $[\dot{h_2}]_{k,u_2} \cdot \psi$ differs from $[\dot{h_2}]_{k,u_1}\cdot \psi$ by a scalar multiple in $\C$. Hence the conclusion follows.
\end{proof}

Now we define modular forms and cusp forms of weight $k \in \frac{1}{2}\Z$ under $\Gamma$ for any congruence subgroup $\Gamma \subseteq \Gamma_1(4)$.
\begin{defi}\label{defi half integral modular forms}
   Let $k \in \frac{1}{2}\Z$ and $\Gamma \subseteq \Gamma_1(4)$ be a congruence subgroup. A function $\psi \in \Omega(\pazocal{H})$ is a \textit{modular form of weight $k$ under $\Gamma$} if it is holomorphic on $\pazocal{H}$, weakly modular of weight $k$ under $\tilde{\Gamma}$, and holomorphic at all the cusps. 

It is a \textit{cusp form} of weight $k$ under $\Gamma$ if it is zero at all the cusps. Let $M_k(\Gamma)$ and $S_k(\Gamma)$ denote the space of modular and cusp forms of weight $k$ under $\Gamma$.
\end{defi}

\begin{rem}
   For $k \in \Z$ and $g \in \Gamma \subseteq \Gamma_1(4)$, we have $[\tilde{g}]_k= [g]_k$.
   Thus the cusp-holomorphy conditions here agree with Definition \ref{cusp definition} and extend it to half-integral weights.
\end{rem}

In what follows we restrict to subgroups of $\Gamma_1(4)$. For $k \in \frac{1}{2}\Z$ (integral or half-integral), and $g \in \GL_2^+(\Q)$, we consistently use the operators $[\dot{g}]_{k,u}$ (and $[\tilde{g}]_{k,u}$ if $g \in \Lambda$) for some $u \in \C$ where $\dot{g} \in \pazocal{G}$.

\subsection{Weak modularity and equivariance}\label{subsection Weak modularity and equivariance}
We will now start the proof of Theorem \ref{modular form} by showing that for $d \in \N^*$, any prime $p$ and any $f:(\Z/p\Z)^d \to \C$, the weighted theta function $\theta_f$ defined by (\ref{theta z^d}) is weakly modular under some congruence subgroup of $\Gamma_1(4)$. For this purpose, we introduce the following family of weighted theta functions. For $j \in X := \mathbb{F}_p \cup \{ \infty \}$, we identify each element with an element in $\mathbb{P}^1_{\mathbb{F}_p}$ as in the discussion before Definition \ref{cusp definition}. In addition, by abuse of notation we identify an element of $\Z/p\Z \simeq \mathbb{F}_p \subset \mathbb{P}^1_{\mathbb{F}_p}$ with its unique representative in $\{0,1,\ldots,p-1\}$ in (\ref{thetafunction1}). Let us define $\theta_f^j: \pazocal{H} \to \C$ to be
\begin{align}
    \theta^j_f(\tau) &:= \sum_{x \in \Z^d} f(x[p]) e^{2\pi i Q(x,x) \frac{\tau -j}{p^2}}\quad\text{for }0 \leq j \leq p-1 , \quad \tau \in \pazocal{H}\label{thetafunction1}\\
    \theta^\infty_f (\tau)&:= p^d \theta_{\mathcal{F}(f)}(\tau) = p^d \sum_{x \in \Z^d} \mathcal{F}(f)(x[p]) e^{2\pi i Q(x,x) \tau},\quad \tau \in \pazocal{H}
\end{align}
where $\mathcal{F}(f)(\xi) = \sum_{x \in (\Z/p\Z)^d} f(x) e^{-2\pi i Q(x,\xi)\frac{1}{p}}$ as defined in (\ref{convention}).

Note that if $f$ is odd, meaning that $f(-x) = -f(x)$ for $x \in (\Z/p\Z)^d$, then we have $\theta_f^j = 0$ for all $j \in X$, so that in this subsection, we will only work with even functions $f$. For any $N \in \N$, let us denote by $V_N$ the space $V_N:= \{ f: (\Z/N\Z)^d \to \C \text{ with }f \text{ even}\}$.

With the help of the notations above, we present a specific example from the Poisson summation formula.
\begin{prop}\label{lemma poisson infinity}
Let $d \in \N^*$, $p$ be a prime and $f \in V_p$. The weighted theta function $\theta^\infty_f$ satisfies the formula
\begin{align}\label{poisson infinity}
    \theta_f^{\infty}(\tau) = \rbra{\frac{i}{2\tau}}^{\frac{d}{2}} \theta^0_{f}\rbra{-\frac{1}{4\tau}},\quad \tau \in \pazocal{H}.
\end{align}
Recall that the square root $\sqrt{\frac{i}{2\tau}}$ is taken in the principal branch, and for $d \geq 1$, by $\rbra{\frac{i}{2\tau}}^{\frac{d}{2}}$ we mean $\rbra{\sqrt{\frac{i}{2\tau}}}^d$.

\end{prop}
Recall that we follow the convention of the Fourier transform formula (\ref{convention}). The proof of the above proposition relies on the following properties. Their standard proofs can be found in \cite{MR1970295}.
\begin{lemma}\label{lemma gaussian}
For $x \in \R^d$, let $G(x) = e^{-\pi Q(x,x)}$ be the Gaussian function on $\R^d$.
\begin{enumerate}
    \item For $d =1$, the Fourier transform of the Gaussian function is itself.
    \item If $\phi \in S(\R^d)$ takes the form of a product $\phi(x) = \phi_1(x_1) \cdots \phi_d(x_d)$, then its Fourier transform is the corresponding product $\mathcal{F}(\phi)(\xi) = \mathcal{F}(\phi_1)(\xi_1)\cdots \mathcal{F}(\phi_d)(\xi_d)$. In particular, (1) holds for any $d \geq 1$.
    \item For any $\phi \in S(\R^d)$ and given any $r >0$, the Fourier transform of the function $\phi(xr)$ is $r^{-d}\mathcal{F}(\phi)(\xi/r)$. In particular, for $r >0$ the Fourier transform of $G(x\sqrt{2r})$ is $(2r)^{-\frac{d}{2}} G(\xi/\sqrt{2r})$.
\end{enumerate}
\end{lemma}

For any $\tau \in \pazocal{H}$ and $t \in \R^d$, set $\eta_{\tau}(t) := e^{2\pi i Q(t,t)\tau}$. Let $\tau = ir$ for $r>0$, then by Lemma \ref{lemma gaussian} we have $\eta_{ir}(t) = G(t\sqrt{2r})$  and its Fourier transform is given by
\begin{align*}
    \mathcal{F}(\eta_{ir})(\xi) &= (2r)^{-\frac{d}{2}} G(\xi/\sqrt{2r})= (2r)^{-\frac{d}{2}} e^{-\pi Q(\xi,\xi)/2r},\quad \xi \in \R^d.
\end{align*}
By analytic continuation, this relation extends to all $\tau \in \pazocal{H}$ and we have
\begin{align}\label{fourier transform gaussian}
    \mathcal{F}(\eta_\tau)(\xi) = \rbra{\frac{i}{2\tau}}^{\frac{d}{2}} e^{2\pi iQ(\xi,\xi) \rbra{-\frac{1}{4\tau}}},\quad \xi \in \R^d.
\end{align}
\begin{proof}[Proof of Proposition \ref{lemma poisson infinity}]
Fix any $\tau \in \pazocal{H}$. For $f \in V_p$, let
\begin{align*}
    \phi(s,t):= \mathcal{F}(f)(s) \eta_{\tau}(t), \quad (s,t) \in (\Z/p\Z)^d \times \R^d,
\end{align*}
which we consider as a function in $(s,t)$ defined on $(\Z/p\Z)^d \times \R^d$. Since the modulus of $\phi(s,t)$ and all its derivatives decrease quickly as $\abs{t}$ grows, it follows that $\phi \in S((\Z/p\Z)^d \times \R^d)$. By Lemma \ref{lemma gaussian}(2), Corollary \ref{corollary poisson}, the Fourier transform formula (\ref{fourier r,zp}), (\ref{fourier transform gaussian}) and the fact that $f$ is even hence $\mathcal{F}^2(f)=f$ we have
\begin{align*}
    \theta_f^{\infty}(\tau)&= p^d \sum_{t \in \Z^d } \phi(t[p],t) = \sum_{\xi \in \Z^d} \mathcal{F}(\phi)(-\xi[p],p^{-1}\xi)\\
    &= \rbra{\frac{i}{2\tau}}^{\frac{d}{2}} \sum_{\xi \in \Z^d} f(\xi[p])e^{2\pi iQ(\xi,\xi) \frac{1}{p^2}\rbra{-\frac{1}{4\tau}}} = \rbra{\frac{i}{2\tau}}^{\frac{d}{2}} \theta^0_{f}\rbra{-\frac{1}{4\tau}},
\end{align*}
as desired.
\end{proof}

Proposition \ref{lemma poisson infinity} carries significant information, as we can reinterpret this formula in the language of group actions, thereby establishing certain weakly modular properties in the following sense.

\subsubsection{The case of the even prime 2}
Now, we begin our discussion of the weak modular property of $\theta_f$ for the case of the even prime 2, which is described in the following theorem. Recall that we have defined $V_2 = \{f: (\Z/2\Z)^d \to \C \text{ with }f\text{ even}\}$.
\begin{thm}\label{p=2 weakly modular}
    Let $d \in \N^*$. There exists a linear representation $\pi_2: \Gamma_1(4) \to \GL(V_2)$ such that for all $g \in \Gamma_1(4)$ and $f \in V_2$ the following holds
    \begin{align}\label{thm p=2 weakly modular eq}
        [\tilde{g}]_{\frac{d}{2}} \cdot \theta_f = \theta_{\pi_2(g)f}.
    \end{align}
\end{thm}

Recall that we have defined the standard symmetric bilinear form on $(\Z/2\Z)^d$ by $Q: (\Z/2\Z)^d \times (\Z/2\Z)^d \to \Z/2\Z$, $(x,y) \mapsto x_1y_1 + \cdots + x_dy_d$, so one can readily verify that the function $x \mapsto e^{-2\pi i Q(x,x)\frac{1}{4}}$ is 2-periodic on $\Z^d$, and thus can be used to construct an operator in $\GL(V_2)$. Precisely, let $f \in V_2$, we call
\begin{align}
    M: V_2 \to V_2, \quad f(x) \mapsto e^{-2\pi i Q(x,x)\frac{1}{4}} f(x)
\end{align}
the multiplication operator in $\GL(V_2)$.

Recall $\alpha =\begin{bsmallmatrix}1&1\\0&1\end{bsmallmatrix}$ so that for any $k \in \frac{1}{2}\Z$ and $u \in \C$, $[\tilde{\alpha}]_{k,u}$ defines the same operator on $\Omega(\pazocal{H})$, i.e. for $\psi \in \Omega(\pazocal{H})$ we have $([\tilde{\alpha}]_{k,u}\cdot\psi)(\tau)=\psi(\alpha^{-1}\cdot\tau)$, so we simply write $[\tilde{\alpha}]$ instead of $[\tilde{\alpha}]_{k,u}$ in the sequel. For $d \in \N^*$, let $u$ be a solution of the equation $\rbra{\frac{i}{2}}^{\frac{d}{2}} = \rbra{\frac{1}{4}}^u$.

Thanks to Proposition \ref{lemma poisson infinity}, which is based on the Poisson summation formula, we obtain the following result.
\begin{prop}\label{prop transform p=2}
    Let $d \in \N^*$, $f \in V_2$ and $\theta_f$ as defined in (\ref{theta z^d}), then the action of $[\tilde{\alpha}]$ and $[\tilde{\beta}]_{\frac{d}{2}}$ on $\theta_f$ are as follows
    \begin{align}
        [\tilde{\alpha}] \cdot \theta_f &= \theta_f, \label{p=2 alpha}\\
        [\tilde{\beta}]_{\frac{d}{2}} \cdot \theta_f &= \theta_{Mf},\label{p=2 beta}
    \end{align}
where the second formula is derived by using the following two actions of $[\tilde{\gamma}]_{\frac{d}{2},u}$ on the family of weighted theta functions $\theta_f^j$ for $j = 0$ and $\infty$
\begin{align}
    [\tilde{\gamma}]_{\frac{d}{2},u} \cdot \theta^0_{f} &=\theta_f^{\infty},  \\
[\tilde{\gamma}]_{\frac{d}{2},u} \cdot \theta^{\infty}_f &=\theta^0_{f}.
\end{align}
\end{prop}
\begin{proof}
Both formula (\ref{p=2 alpha}) and (\ref{p=2 beta}) need the help of the family of weighted theta functions $\theta_f^j$ for $j = 0$ and $\infty$ for their derivations. We can readily compute that for $\tau \in \pazocal{H}$ we have
\begin{align*}
    \rbra{[\tilde{\alpha}] \cdot \theta_f^0} (\tau) &= \theta_f^0(\tau-1)\\
    &= \sum_{x \in \Z^d} Mf(x[2]) e^{2\pi i Q(x,x)\frac{\tau}{4}} = \theta^0_{Mf}(\tau)
\end{align*}
and $\rbra{[\tilde{\alpha}]\cdot \theta_f^{\infty}}(\tau) = \theta_f^{\infty}(\tau)$.

By Proposition \ref{prop cocycle h bar}, we obtain that $\bar{h}(\gamma^{-1},\tau) = \bar{h}(\gamma,\gamma^{-1}\cdot \tau)^{-1} = \sqrt{\tau}$ and by Proposition \ref{lemma poisson infinity} we know that for $\tau \in \pazocal{H}$ and $d \in \N^*$, the formula (\ref{poisson infinity}) holds. Since $u$ is a solution of the equation $\rbra{\frac{i}{2}}^{\frac{d}{2}} = \rbra{\frac{1}{4}}^u$, it follows that
\begin{align*}
    \theta_f^{\infty}(\tau) &= \tau^{-d/2}4^{-u} \theta^0_{f}(\gamma^{-1} \cdot \tau)\nonumber\\
    &=  \rbra{[\tilde{\gamma}]_{\frac{d}{2},u} \cdot \theta^0_{f}}(\tau).
\end{align*}
A symmetric formula by sending $\tau \mapsto -\frac{1}{4\tau}$ in (\ref{poisson infinity}) gives
\begin{align*}
    \theta_f^{\infty}\rbra{-\frac{1}{4\tau}} = \rbra{\frac{2\tau}{i}}^{\frac{d}{2}} \theta^0_{f}(\tau).
\end{align*}
Equivalently, we have
\begin{align*}
   \rbra{[\tilde{\gamma}]_{\frac{d}{2},u} \cdot \theta^{\infty}_f }(\tau)= \theta^0_{f}(\tau).
\end{align*}

By definition we know $\theta_f = 2^{-d}\theta_{\mathcal{F}(f)}^{\infty}$, this implies that $[\tilde{\alpha}] \cdot \theta_f = \theta_f$ as in (\ref{p=2 alpha}). In addition,
\begin{align*}
    [\tilde{\beta}]_{\frac{d}{2}} \cdot \theta^\infty_f = [\tilde{\gamma}\tilde{\alpha}\widetilde{\gamma^{-1}}]_{\frac{d}{2}} \cdot \theta^\infty_f = [\tilde{\gamma}]_{\frac{d}{2},u} [\tilde{\alpha}] \cdot \theta_f^0 = [\tilde{\gamma}]_{\frac{d}{2},u} \cdot \theta^0_{Mf} = \theta^\infty_{Mf}
\end{align*}
so it follows that $[\tilde{\beta}]_{\frac{d}{2}} \cdot \theta_f = \theta_{Mf}$ as in (\ref{p=2 beta}).
\end{proof}

Next, we will build the linear representation $\pi_2: \Gamma_1(4) \to \GL(V_2)$. For every $g \in \Gamma_1(4)$, we first map it to $\mu(g) := \gamma g \gamma^{-1}$ through the automorphism $\mu: \Gamma_1(4) \to \Gamma_1(4)$ defined as conjugation by $\gamma$. Let $\omega_2: \SL_2(\Z) \to \SL_2(\Z/4\Z)$ be the natural group homomorphism, then the element $\omega_2 \circ \mu(g)$ of $\SL_2(\Z/4\Z)$ belongs to the group $U_4$ defined by
\begin{align}
    U_4 := \bbra{\begin{bmatrix}
        1&i\\0&1
    \end{bmatrix}: i \in \Z/4\Z} \simeq \Z/4\Z.
\end{align}
This group is cyclic and is generated by $\omega_2(\alpha)$. Let $\rho_2: U_4 \to \GL(V_2)$ be the unique representation such that $\rho_2 \circ \omega_2(\alpha) = M$, then we are ready to prove the above theorem.
\begin{proof}[Proof of Theorem \ref{p=2 weakly modular}]
Let $\pi_2: \Gamma_1(4) \to \GL(V_2)$ be the linear representation defined as follows. For $g \in \Gamma_1(4)$, we map it through the composition of $\omega_2$ and $\mu$ and then assign it to an element of $\GL_2(V_2)$ through the representation $\rho_2: U_4 \to \GL(V_2)$. We can check that
\begin{align*}
    \pi_2(\alpha) &= \rho_2 \circ \omega_2 \circ \mu(\alpha) = \rho_2 \circ \omega_2(\beta)= 1,\\
    \pi_2(\beta) &= \rho_2 \circ \omega_2 \circ \mu(\beta) = \rho_2 \circ \omega_2(\alpha) = M.
\end{align*}
By Proposition \ref{prop transform p=2}, we know that (\ref{thm p=2 weakly modular eq}) is verified for $\alpha,\beta$. By Lemma \ref{lemma Lambda inside SL2Z} we know that $\alpha,\beta$ generate the group $\Gamma_1(4)$, so the result follows.
\end{proof}
\subsubsection{The case of any odd prime $p$}
For any odd prime $p$, we denote by $\omega': \Z \to \Z/p\Z$ the canonical projection of $\Z$ on $\Z/p\Z$. One can define a ring homomorphism $\omega'':\Z[\tfrac{1}{2}] \to \Z/p\Z$ as follows. Let $\iota_2 \in \Z/p\Z$ be the inverse of 2 in $\Z/p\Z$ so that
\begin{align*}
    \omega'':\Z[\tfrac{1}{2}] &\to \Z/p\Z\\
    \frac{n}{2^k} &\mapsto \omega'(n)\iota_2^k.
\end{align*}
The ring homomorphism $\omega''$ gives rise to a group homomorphism $\omega: \GL_2(\Z[\tfrac{1}{2}]) \to \GL_2(\mathbb{F}_p)$. Recall that $\Lambda = \bra{\alpha,\gamma}$. We let $\Lambda$ act on $\mathbb{P}^1_{\mathbb{F}_p}$ through the composition of $\omega$ and the natural inclusion $\Lambda \subset \GL_2(\Z[\frac{1}{2}])$. Note that this action is transitive.
\begin{lemma}
    For any odd prime $p$, the action of $\Lambda$ on $\mathbb{P}^1_{\mathbb{F}_p}$ is transitive.
\end{lemma}
\begin{proof}
It suffices to show that there exists some $j \in \mathbb{P}^1_{\mathbb{F}_p}$ such that $\Lambda \cdot j = \mathbb{P}^1_{\mathbb{F}_p}$. Take $j = 0$, then it follows that $\alpha \cdot 0 = 1, \ldots, \alpha^{p-1} \cdot 0 = p-1, \alpha^p \cdot 0= 0$ and $\gamma \cdot 0 = \infty$.
\end{proof}

Recall for any prime $p$, we have defined $V_p= \{ f: (\Z/p\Z)^d \to \C \text{ with }f \text{ even}\}$. The main result we aim to prove in this subsection is the following, which can be interpreted as a form of weak modular property of the family of weighted theta functions $\theta_f^j$ for $j \in X$.
\begin{thm}\label{prop equivariance}
    Let $d \in \N^*$ and $p$ be an odd prime. There exists a cocycle $\sigma: \Lambda \times \mathbb{P}^1_{\mathbb{F}_p} \to \GL(V_p)$ and $u \in \C$ such that for all $g \in \Lambda$, $j \in \mathbb{P}^1_{\mathbb{F}_p}$ and $f \in V_p$ the following holds
    \begin{equation}\label{action of g on theta}
        [\tilde{g}]_{\frac{d}{2},u} \cdot \theta_f^j = \theta^{g\cdot j}_{\sigma(g,j)f}.
    \end{equation}
\end{thm}
\begin{rem}
   The reason we expect the map $\sigma: \Lambda \times \mathbb{P}^1_{\mathbb{F}_p} \to \GL(V_p)$ (if it exists) to be a cocycle lies in its defining structure. We can readily verify that for $g_1,g_2 \in \Lambda$ and $j \in \mathbb{P}^1_{\mathbb{F}_p}$, we have
   \begin{align*}
       [\tilde{g_1}]_{\frac{d}{2},u} \cdot \rbra{ [\tilde{g_2}]_{\frac{d}{2},u} \cdot \theta^j_f} &= [\tilde{g_1}]_{\frac{d}{2},u}\cdot \theta^{g_2 \cdot j}_{\sigma(g_2,j)f} = \theta^{(g_1g_2)\cdot j}_{\sigma(g_1,g_2 \cdot j)\sigma(g_2,j)f},\\
       \rbra{[\tilde{g_1}\tilde{g_2}]_{\frac{d}{2},u} } \cdot \theta_f^j &= \theta^{(g_1g_2)\cdot j}_{\sigma(g_1g_2,j)f}.
   \end{align*}
   By Proposition \ref{prop left action half integral} we know that the map $g \mapsto [\tilde{g}]_{\frac{d}{2},u}$ defines an action of $\Lambda$ on $\Omega(\pazocal{H})$, therefore the left hand sides of the above two quantities are equal, hence so are the right hands, which suggests that the map $\sigma$ is a cocycle by (\ref{cocycle property}).
\end{rem}
The above theorem will directly imply an equivariance property in the following sense. For any odd prime $p$, let $V_p^X$ be the space of $V_p$-valued functions on $X$, i.e. $V_p^X := \bbra{F: X \to V_p}$, which we can write as a direct sum of $W_j$ defined by $W_j := \bbra{F\in V_p^X: F(i)=0 \text{ for }i\neq j } \simeq V_p$ over all $j \in X$. This means
\begin{equation}
 V_p^X = \bbra{F: X \to V_p} \simeq \oplus_{j \in X} W_j.
\end{equation}
By Proposition \ref{prop cocycle defines a representation}, we know that for any $g \in \Lambda$, $F \in V_p^X$ and $j \in X$, the formula
\begin{equation}\label{action on vpx}
    (g \cdot F)(j) := \sigma(g, g^{-1} \cdot j) F(g^{-1} \cdot j)
\end{equation}
defines an action of $\Lambda$ on $V_p^X$. This gives rise to a linear representation $\Pi: \Lambda \to \GL(V_p^X)$.

Let us define the linear map $\Theta: V_p^X \to \Omega(\pazocal{H})$ by $\Theta(F):= \sum_{j \in X} \theta^j_{F(j)}$, then the equivariance property is stated as follows.
\begin{cor}\label{big theta equivariant}
    Let $d \in \N^*$ and $p$ be an odd prime. For the cocycle $\sigma$ and $u \in \C$ as in (\ref{action of g on theta}), we have that for any $g \in \Lambda$ the equality $[\tilde{g}]_{\frac{d}{2},u} \cdot \Theta(F)= \Theta(g \cdot F)$ holds.
\end{cor}
\begin{proof}
Indeed, for any $g \in \Lambda$
    \begin{align*}
    [\tilde{g}]_{\frac{d}{2},u} \cdot \Theta(F) &= \sum_{j \in X} [\tilde{g}]_{\frac{d}{2},u}\cdot \theta^j_{F(j)} = \sum_{j \in X} \theta^{g \cdot j}_{\sigma(g,j)F(j)},\\
    \Theta(g \cdot F) &= \sum_{j \in X} \theta^j_{(g \cdot F)(j)} = \sum_{j \in X} \theta^j_{\sigma (g,g^{-1}\cdot j)F(g^{-1}\cdot j)}= \sum_{j \in X} \theta^{g\cdot j}_{\sigma(g,j)F(j)}.
\end{align*}
Hence $[\tilde{g}]_{\frac{d}{2},u} \cdot \Theta(F)= \Theta(g \cdot F)$.
\end{proof}

We now begin proving Theorem \ref{prop equivariance}. Let $p$ be an odd prime and $f \in V_p$, let
\begin{equation}\label{operatorL}
    L: V_p \to V_p, \quad f(x) \mapsto e^{-2\pi i Q(x,x)\frac{1}{p}} f(x)
\end{equation}
be the multiplication operator by the Gaussian in $\GL(V_p)$, and for every $1 \leq j \leq p-1$ let
\begin{equation}\label{operatorsj}
    S_j: V_p \to V_p, \quad f(x) \mapsto f(jx)
\end{equation}
be the parameter multiplication operator by $j$ in $\GL(V_p)$.

For $d \in \N^*$, recall $u$ is a solution of the equation $\rbra{\frac{i}{2}}^{\frac{d}{2}} = \rbra{\frac{1}{4}}^u$. For every $1 \leq j \leq p-1$, we define $j'$ to be the unique element in $\{1,\ldots,p-1\}$ such that $4jj'+1 \in p\Z$, and we define $k_j := \frac{4jj'+1}{p} \in \Z$. We obtain the following result.
\begin{prop}\label{proposition transformation theta general}
Let $d \in \N^*$, $p$ be an odd prime and $f \in V_p$. The action of $[\tilde{\alpha}]$ and $[\tilde{\gamma}]_{\frac{d}{2},u}$ on the family of weighted functions $\theta_f^j$ for $j \in X$ can be described in Table \ref{table 1}.
\begin{table}[h]
\setlength{\tabcolsep}{10pt} 
\begin{threeparttable}
\captionof{table}{The action of $[\tilde{\alpha}]$ and $[\tilde{\gamma}]_{\frac{d}{2},u}$ on $\theta_f^j$ for $j \in X$}\label{table 1}
\renewcommand{\arraystretch}{1.5}
    \begin{tabular}{l|l|l}
    \toprule
        $j=0$&\multirow{2}{*}{$[\tilde{\alpha}] \cdot \theta_f^j = \theta_f^{j+1}$}&$[\tilde{\gamma}]_{\frac{d}{2},u} \cdot \theta^j_{f}= \theta_{f}^{\infty}$\\
        \cmidrule{1-1}\cmidrule{3-3}
        $1\leq j \leq p-2$&&\multirow{2}{*}{$ [\tilde{\gamma}]_{\frac{d}{2},u} \cdot \theta_{f}^{j} = \theta^{j'}_{S_{-2j'}L^{-k_jj}f}$}\\
        \cmidrule{1-2}
        $j=p-1$&$ [\tilde{\alpha}] \cdot \theta_f^j = \theta_{L f}^0 $& \\
        \midrule
        $j=\infty$&$[\tilde{\alpha}] \cdot \theta_f^j = \theta_f^\infty$&$[\tilde{\gamma}]_{\frac{d}{2},u} \cdot \theta^j_f = \theta^0_{f}$\\
        \bottomrule
    \end{tabular}
\end{threeparttable}
\end{table}
\end{prop}
\begin{proof}
We can readily compute that for $\tau \in \pazocal{H}$ and $0\leq j \leq p-2$, we have
\begin{align}
    \rbra{[\tilde{\alpha}] \cdot \theta_f^j}(\tau) = \theta_f^j(\tau-1)=\theta_f^{j+1}(\tau)
\end{align}
and for $j = p-1$, we have
\begin{align}
    \rbra{[\tilde{\alpha}] \cdot \theta_f^{p-1}}(\tau) &= \sum_{x \in \Z^d} f(x[p]) e^{2\pi i Q(x,x) \frac{\tau -p}{p^2}}\nonumber\\
    &= \sum_{x \in \Z^d} (L f)(x[p]) e^{2\pi i Q(x,x) \frac{\tau}{p^2}} = \theta_{L f}^0 (\tau).
\end{align}
Finally for $j = \infty$, we have
\begin{align}
    \rbra{[\tilde{\alpha}] \cdot \theta_f^\infty}(\tau) = \theta_f^\infty(\tau).
\end{align}
Therefore, the first column of Table \ref{table 1} is true. Next we compute the second column. By Proposition \ref{prop transform p=2}, we know that the first and last line are true. For $1\leq j \leq p-1$, the expression (\ref{thetafunction1}) can be rewritten as
\begin{align*}
    \theta_f^j(\tau) = \sum_{x \in \Z^d} (M_j f)(x[p^2]) e^{2\pi i Q(x,x) \frac{\tau}{p^2}},
\end{align*}
where we have denoted by $M_j: V_p \to V_{p^2}$ the linear operator such that for $f \in V_p$ and $x \in (\Z/p^2\Z)^d$, we have
\begin{equation}
   M_jf(x) = e^{-2\pi i Q(x,x)\frac{j}{p^2}} f(x).
\end{equation}
We will apply Poisson summation formula to $\theta_f^j$. To this aim, we first compute the Fourier transform of $M_jf$. Take $d=1$. For any $f \in V_p$, the Fourier transform of $M_j f$ is given by
\begin{align*}
    \mathcal{F}_{\Z/p^2\Z}(M_j f)(\xi) &= \sum_{x \in \Z/p^2\Z} e^{-2\pi i Q(x,x)\frac{j}{p^2}} f(x) e^{-2\pi i Q(x,\xi)\frac{1}{p^2}}\\
    &=\sum_{k =0}^{p-1} f(k) \sum_{l \in \Z/p\Z}e^{-2\pi i  \frac{1}{p^2}\sbra{Q(k+pl,k+pl)j+Q(k+pl,\xi)}}\\
    &= \sum_{k =0}^{p-1} f(k) e^{-2\pi i \frac{1}{p^2}\sbra{Q(k,k)j+Q(k,\xi)}}\sum_{l \in \Z/p\Z} e^{-2\pi i \frac{1}{p}Q(l,2jk+\xi)}.
\end{align*}
The second summation is nonzero only when $2jk + \xi \in p \Z$. In this case, $k=2j' \xi[p]$ where $j'$ is the unique element in $\{1,\ldots,p-1\}$ such that $4jj'+1 \in p \Z$. Then, the exponential term in the first summation gives
\begin{align*}
   e^{-2\pi i \frac{1}{p^2}\sbra{Q(k,k)j+Q(k,\xi)}}&= e^{-2\pi i \frac{1}{p^2}\sbra{4jj'^2Q(\xi,\xi)+2j'Q(\xi,\xi)}}= e^{-2\pi i \frac{1}{p^2}\sbra{(4jj'+1)j'+j'}Q(\xi,\xi)}.
\end{align*}
Written in terms of operators $S_j$ and $L$ we have
\begin{align*}
    \mathcal{F}_{\Z/p^2\Z}(M_j f)(\xi) &= p f(2j' \xi)e^{-2\pi i \frac{j'}{p^2}Q(\xi,\xi)}e^{-2\pi i \frac{(4jj'+1)j'}{p^2}Q(\xi,\xi)}= p (M_{j'}L^{k_jj'}S_{2j'}f)(\xi),
\end{align*}
where we recall that $k_j = \frac{4jj'+1}{p} \in \Z$. By Lemma \ref{lemma gaussian}(2), we conclude that for any $d\geq 1$, the following holds
\begin{align*}
    \mathcal{F}_{(\Z/p^2\Z)^d}(M_j f)(\xi) = p^d (M_{j'}L^{k_jj'}S_{2j'}f)(\xi).
\end{align*}
Back to our goal using Poisson summation, by Corollary \ref{corollary poisson} and the Fourier transform formula (\ref{fourier r,zp}), (\ref{fourier transform gaussian}) we have
\begin{align*}
    \theta^j_f(\tau) &= \rbra{\frac{i}{2 \frac{\tau}{p^2}}}^{\frac{d}{2}} \frac{1}{p^{2d}} \sum_{\xi \in \Z^d} \mathcal{F}_{(\Z/p^2\Z)^d}(M_j f)(-\xi[p^2]) e^{2\pi i Q(\xi,\xi)\frac{1}{p^4} \rbra{-\frac{p^2}{4\tau}}}\\
    &= \rbra{\frac{i}{2\tau}}^{\frac{d}{2}} \sum_{\xi \in \Z^d}(M_{j'}L^{k_jj'}S_{2j'}f)(\xi[p]) e^{2\pi i Q(\xi,\xi)\frac{1}{p^2} \rbra{-\frac{1}{4\tau}}}\\
    &= \rbra{\frac{i}{2\tau}}^{\frac{d}{2}} \theta_{L^{k_jj'}S_{2j'}f}^{j'}\rbra{-\frac{1}{4\tau}}.
\end{align*}
Equivalently,
\begin{align*}
    \rbra{[\tilde{\gamma}]_{\frac{d}{2},u} \cdot \theta_{L^{k_jj'}S_{2j'}f}^{j'}}(\tau) = \theta^j_f(\tau),
\end{align*}
and since $j \mapsto j'$ is an involution on $\{1,\ldots,p-1\}$, we get
\begin{align}
    \rbra{[\tilde{\gamma}]_{\frac{d}{2},u} \cdot \theta_{f}^{j}}(\tau) &= \theta^{j'}_{\sbra{L^{k_jj}S_{2j}}^{-1}f}(\tau)=\theta^{j'}_{S_{-2j'}L^{-k_jj}f}(\tau).
    \end{align}
\end{proof}

Table \ref{table 1} suggests that for any odd prime $p$, there should exist a cocycle $\sigma: \Lambda \times \mathbb{P}_{\mathbb{F}_p}^1 \to \GL(V_p)$ such that (\ref{action of g on theta}) holds. In Table \ref{table3} we rewrite what the values of $\sigma$ should be for $\alpha, \gamma$.

Let
\[
    \Lambda_{\infty}
    =
    \{g\in \Lambda : g\cdot \infty=\infty\}
\]
be the stabilizer of \(\infty=[1:0]\in \mathbb P^1_{\mathbb F_p}\), where
\(\Lambda\) acts on \(\mathbb P^1_{\mathbb F_p}\) by reduction modulo \(p\).
Let
\[
    \Delta :=
    \left\{
    g=\begin{bmatrix}
        a&b\\ c&d
    \end{bmatrix}
    :
    a,b,d\in \mathbb Z\left[\tfrac{1}{2}\right],
    \ c\in p\mathbb Z\left[\tfrac{1}{2}\right],
    \ \det g\in 4^{\mathbb Z}
    \right\}.
\]
Then \(\Delta\) is a subgroup of \(\GL_2(\mathbb Z[\frac{1}{2}])\). We have the
following description of \(\Lambda_{\infty}\).
\begin{lemma}\label{lemma description of lambda infty}
For any odd prime \(p\), one has
\[
    \Lambda_{\infty}=\Lambda\cap \Delta.
\]
\end{lemma}
\begin{proof}
Let \(g\in \Lambda\cap \Delta\), and write $g=
    \begin{bsmallmatrix}
        a&b\\ c&d
    \end{bsmallmatrix}$.
Since \(g\in \Delta\), we have \(c\in p\mathbb Z[\tfrac12]\), so $\bar c=0$
after reduction modulo \(p\). Moreover, \(\det g\in 4^{\mathbb Z}\), and since
\(p\) is odd, its reduction modulo \(p\) is nonzero. Thus
\[
    0\neq \overline{\det g}
    =\bar a\bar d-\bar b\bar c
    =\bar a\bar d.
\]
In particular, \(\bar a\neq 0\). Therefore
\[
    g\cdot \infty
    =
    g\cdot [1:0]
    =
    [\bar a:\bar c]
    =
    [\bar a:0]
    =
    [1:0]
    =
    \infty.
\]
Hence \(g\in \Lambda_\infty\). This proves $\Lambda\cap \Delta \subseteq \Lambda_{\infty}.$

Conversely, let \(g\in \Lambda_{\infty}\). Write $
    g=\begin{bsmallmatrix}
        a&b\\ c&d
    \end{bsmallmatrix}
    \in \GL_2\left(\mathbb Z\left[\tfrac{1}{2}\right]\right).
$ Since \(g\cdot \infty=\infty\) in \(\mathbb P^1_{\mathbb F_p}\), we have $ [\bar a:\bar c]=[1:0].$ Therefore \(\bar c=0\), which means $c\in p\mathbb Z\left[\tfrac{1}{2}\right]$. Moreover, since \(\Lambda\) is generated by \(\alpha\) and \(\gamma\), with $\det(\alpha)=1$ and $ \det(\gamma)=4$, we have \(\det(g)\in 4^{\mathbb Z}\). Hence, \(g\in \Delta\). Since also
\(g\in \Lambda\), we get $g\in \Lambda\cap \Delta.$ Thus, $\Lambda_{\infty}\subseteq \Lambda\cap \Delta.$
\end{proof}
We claim the refined version of Theorem \ref{prop equivariance} as follows.
\begin{thm}\label{prop theta tilde}
    Let $d \in \N^*$ and $p$ be an odd prime. There exists a cocycle $\tilde{\varphi}: \Lambda \times \mathbb{P}_{\mathbb{F}_p}^1 \to \Delta$ and a linear representation $\pi_{\Delta}: \Delta \to \GL(V_p)$, so that the composition of $\pi_{\Delta}$ and $\tilde{\varphi}$ defines another cocycle $\sigma: \Lambda \times \mathbb{P}_{\mathbb{F}_p}^1 \to \GL(V_p)$, and there exists $u \in \C$ such that for any $g \in \Lambda$, $j \in \mathbb{P}_{\mathbb{F}_p}^1$ and $f \in V_p$ the equality (\ref{action of g on theta}) in Theorem \ref{prop equivariance} holds, with the values of $\sigma$, for $\alpha, \gamma$, aligned with those in Table \ref{table3}.
\end{thm}
\begin{table}[h]
\setlength{\tabcolsep}{10pt} 
\renewcommand{\arraystretch}{1.5}
\begin{minipage}{.5\linewidth}
\begin{threeparttable}
\captionof{table}{The values of $\sigma(g,j) \in \GL(V_p)$ for $\alpha,\gamma$ and $j \in X$}\label{table3}
    \begin{tabular}{l|l|l}
    \toprule
    &$g=\alpha$&$g=\gamma$\\
    \midrule
        $j=0$&\multirow{2}{*}{1}&$1$\\
        \cmidrule{1-1}\cmidrule{3-3}
        $1\leq j \leq p-2$&&\multirow{2}{*}{$S_{-2j'}L^{-k_jj}$}\\
        \cmidrule{1-2}
        $j=p-1$&$L$& \\
        \midrule
        $j=\infty$&$1$&$1$\\
        \bottomrule
    \end{tabular}
    \end{threeparttable}
    \end{minipage}
  \begin{minipage}{.4\linewidth}
  \begin{threeparttable}
   \captionof{table}{The values of $\tilde{\varphi}(g,j)\in \Delta$ for $\alpha,\gamma$ and $j \in X$}\label{table4}
    \begin{tabular}{l|l|l}
    \toprule
    &$g=\alpha$&$g=\gamma$\\
    \midrule
        $j=0$&\multirow{2}{*}{1}&$\begin{bsmallmatrix}
         -4&0\\0&-4
        \end{bsmallmatrix}$\\
        \cmidrule{1-1}\cmidrule{3-3}
        $1\leq j \leq p-2$&&\multirow{2}{*}{$\begin{bsmallmatrix}
         4j'&-4k_j\\p&-4j
        \end{bsmallmatrix}$}\\
        \cmidrule{1-2}
        $j=p-1$&$\begin{bsmallmatrix}
         1&-4\\0&1
        \end{bsmallmatrix}$& \\
        \midrule
        $j=\infty$&$\begin{bsmallmatrix}
         1&0\\p&1
        \end{bsmallmatrix}$&$1$\\
        \bottomrule
    \end{tabular}
    \end{threeparttable}
    \end{minipage}
\end{table}

For any odd prime $p$, we first build the linear representation $\pi_{\Delta}: \Delta \to \GL(V_p)$. We define a homomorphism $\xi : \Delta \to \Z$ as follows: for every $g \in \Delta$, let $\xi(g)$ be a unique integer such that $4^{\xi(g)} = \det g$. Then the element $\omega(2^{-\xi(g)}g)$ of $\SL_2(\mathbb{F}_p)$ belongs to the group $P$ defined by
\begin{align}
    P &:= \bbra{\begin{bmatrix}
       j&i\\0&j^{-1}
    \end{bmatrix}:j \in \mathbb{F}_p^*, i\in\mathbb{F}_p}= A \ltimes U,
\end{align}
which is the semidirect product of $A$ by $U$, where
\begin{align*}
    A &:= \bbra{\begin{bmatrix}
        j&0\\0&j^{-1}
    \end{bmatrix}: j \in \mathbb{F}_p^*}, U :=\bbra{\begin{bmatrix}
       1&i\\0&1
    \end{bmatrix}: i \in \mathbb{F}_p} \simeq \mathbb{F}_p.
\end{align*}
This means that $P = AU$ and that $A \cap U = \{1\}$. We write the elements of $A$ as $a_j:=\begin{bsmallmatrix}
        j&0\\0&j^{-1}
    \end{bsmallmatrix}\text{ for $j \in \mathbb{F}_p^*$}$. The group $U$ is cyclic and is generated by $\omega(\alpha)$. The representation of $P$ can be constructed from those of $A$ and $U$.
\begin{prop}\label{representation P}
Let $p$ be an odd prime and $\iota_{-4}$ be the inverse of $-4$ in $\mathbb{F}_p$. For $j \in \mathbb{F}_p^*$ we have the following identity in $\GL(V_p)$,
\begin{equation}\label{semidirect relation}
    S_j L S_j^{-1} = L^{j^2}.
\end{equation}
In particular, there exists a unique representation $\rho: P \to \GL(V_p)$ such that
\begin{align}
    \rho(a_j) = S_j \quad \text{ for }j \in \mathbb{F}_p^*, \quad (\rho\circ \omega)(\alpha) = L^{\iota_{-4}}.
\end{align}
\begin{proof}
 We first prove \eqref{semidirect relation}. Since \(S_j^{-1}=S_{j^{-1}}\),
for \(f\in V_p\) and \(x\in (\mathbb Z/p\mathbb Z)^d\), we have
\begin{align*}
    (S_jLS_j^{-1}f)(x) = (LS_j^{-1}f)(jx) = e^{-2\pi i Q(jx,jx)\frac{1}{p}} f(x) = e^{-2\pi i Q(x,x)\frac{j^2}{p}} f(x) = (L^{j^2}f)(x).
\end{align*}
Thus, $S_jLS_j^{-1}=L^{j^2}.$
    
Now write $u_t :=
    \begin{bsmallmatrix}
        1&t\\0&1
    \end{bsmallmatrix}$ for $t\in \mathbb F_p$. Thus \(\omega(\alpha)=u_1\). Since \(L^p=\mathrm{id}\), the assignment $\rho_2(u_t):=L^{\iota_{-4}t}$ defines a representation \(\rho_2:U\to \GL(V_p)\). Similarly, the assignment $\rho_1(a_j):=S_j$ defines a representation \(\rho_1:A\to \GL(V_p)\), since
\(S_jS_k=S_{jk}\).

It remains to check compatibility with the semidirect product structure. For
\(j\in \mathbb F_p^*\) and \(t\in \mathbb F_p\), we have
\[
    a_j u_t a_j^{-1}=u_{j^2t}.
\]
Using \eqref{semidirect relation}, we get
\begin{align*}
    \rho_1(a_j)\rho_2(u_t)\rho_1(a_j)^{-1}= S_j L^{\iota_{-4}t}S_j^{-1} = L^{\iota_{-4}j^2t} = \rho_2(u_{j^2t}) = \rho_2(a_j u_t a_j^{-1}).
\end{align*}
Therefore \(\rho_1\) and \(\rho_2\) are compatible with the semidirect product
decomposition \(P=A\ltimes U\).

Hence we may define, for the unique decomposition \(a_j u_t\in P\),
\[
    \rho(a_j u_t):=\rho_1(a_j)\rho_2(u_t)
    =S_jL^{\iota_{-4}t}.
\]
The compatibility relation above implies that \(\rho\) is multiplicative, and
therefore \(\rho:P\to \GL(V_p)\) is a representation. Moreover, $\rho(a_j)=S_j$ and $ (\rho\circ \omega)(\alpha)=\rho(u_1)=L^{\iota_{-4}}.$

Finally, uniqueness follows from the unique decomposition \(P=AU\) with
\(A\cap U=\{1\}\). Indeed, \(U=\langle \omega(\alpha)\rangle\), so the
prescribed values of \(\rho\) on \(A\) and on \(\omega(\alpha)\) determine
\(\rho\) on every element \(a_j u_t\in P\). This proves the claim.
\end{proof}
\end{prop}

Next, we will describe how to build the cocycle $\tilde{\varphi}: \Lambda \times \mathbb{P}_{\mathbb{F}_p}^1 \to \Delta$. We choose a section $S: \mathbb{P}^1_{\mathbb{F}_p} \to \Lambda$ of the surjection $\Lambda \to \Lambda / \Lambda_\infty \simeq \mathbb{P}^1_{\mathbb{F}_p}$ as follows,
\begin{align}\label{section}
    S: \mathbb{P}^1_{\mathbb{F}_p} &\to \Lambda\\\nonumber
    j &\mapsto \alpha^j \gamma \quad\text{ for }0\leq j \leq p-1,\\
    \infty&\mapsto 1.\nonumber
\end{align}
This defines a cocycle $\varphi: \Lambda \times \mathbb{P}^1_{\mathbb{F}_p} \to \Lambda_{\infty}$ by the formula $gS(j)=S(g \cdot j)\varphi(g,j)$ for $g \in \Lambda$ and $j \in \mathbb{P}^1_{\mathbb{F}_p}$. Denote by $\pazocal{P}$ the matrix $\begin{bsmallmatrix}0&1\\p&0\end{bsmallmatrix}$, then one can quickly verify that $\Delta$ is stable under conjugation by $\pazocal{P}$. We define $\tilde{\varphi}: \Lambda \times \mathbb{P}^1_{\mathbb{F}_p} \to \Delta$ by $\tilde{\varphi}(g,j):= \pazocal{P}\varphi(g,j)\pazocal{P}^{-1}$ for $g \in \Lambda$ and $j \in \mathbb{P}^1_{\mathbb{F}_p}$.
\begin{lemma}
Let $p$ be an odd prime. The map $\tilde{\varphi}: \Lambda \times \mathbb{P}^1_{\mathbb{F}_p} \to \Delta$ is a cocycle, with values, for generators $\alpha, \gamma$ summarized in Table \ref{table4}.
\end{lemma}
\begin{proof}
For $g_1,g_2 \in \Lambda$ and $j \in \mathbb{P}^1_{\mathbb{F}_p}$, we have
\begin{align*}
    \varphi(g_1g_2,j) &=  S(g_1g_2\cdot j)^{-1}g_1g_2S(j)\\
    &=S(g_1\cdot (g_2\cdot j))^{-1}g_1S(g_2 \cdot j)S(g_2 \cdot j)^{-1}g_2S(j)\\
    &=\varphi(g_1,g_2\cdot j)\varphi(g_2,j),
\end{align*}
so the map $\varphi$ is a cocycle and it follows that so is the map $\tilde{\varphi}$. We first compute explicitly the values of $\varphi(g,j)$ for $g = \alpha,\gamma$ and $j \in \mathbb{P}^1_{\mathbb{F}_p}$ as follows. For $g = \alpha$ we have
\begin{align*}
    \varphi(\alpha,j) &= S(j+1)^{-1}\alpha S(j) = 1 \quad\text{for }0 \leq j \leq p-2,\\
    \varphi(\alpha,p-1) &=  S(0)^{-1}\alpha S(p-1)= \gamma^{-1} \alpha \alpha^{p-1} \gamma = \begin{bsmallmatrix}
      1&0\\-4p&1
    \end{bsmallmatrix},\\
    \varphi(\alpha,\infty)&= S(\infty)^{-1}\alpha S(\infty)= \alpha.
\end{align*}
For $g = \gamma$ we have
\begin{align*}
    \varphi(\gamma,0) &= S(\infty)^{-1}\gamma S(0) =  \gamma^2\pazocal = \begin{bsmallmatrix}
      -4&0\\0&-4
    \end{bsmallmatrix}.\\
    \varphi(\gamma,j) &=  S(\gamma \cdot j)^{-1}\gamma S(j) = S(j')^{-1}\gamma S(j) \\ &= \gamma^{-1} \alpha^{-j'}\gamma\alpha^j \gamma = \begin{bsmallmatrix}
      -4j&1\\-4pk_j&4j
    \end{bsmallmatrix}  \quad\text{for }1\leq j \leq p-1,
\end{align*}
where $j'$ is the unique element in $\{1,\ldots,p-1\}$ such that $4jj'+1 \in p\Z$, with $k_j = \frac{4jj'+1}{p}$ as defined before. Finally,
\begin{align*}
    \varphi(\gamma,\infty) = S(0)^{-1}\gamma S(\infty) = 1.
\end{align*}

We know for $g \in \Lambda$ and $j \in \mathbb{P}^1_{\mathbb{F}_p}$ that $\tilde{\varphi}(g,j) = \pazocal{P}\varphi(g,j)\pazocal{P}^{-1}$. Correspondingly we conclude that for $g = \alpha$ we have
\begin{align*}
    \tilde{\varphi}(\alpha,j) &= 1 \quad\text{for }0 \leq j \leq p-2,\\
    \tilde{\varphi}(\alpha,p-1) &= \begin{bsmallmatrix}
      1&-4\\0&1
    \end{bsmallmatrix},\\
    \tilde{\varphi}(\alpha,\infty)&= \begin{bsmallmatrix}
      1&0\\p&1
    \end{bsmallmatrix}.
\end{align*}
For $g = \gamma$ we have
\begin{align*}
    \tilde{\varphi}(\gamma,0) &= \begin{bsmallmatrix}
      -4&0\\0&-4
    \end{bsmallmatrix}.\\
    \tilde{\varphi}(\gamma,j) &= \begin{bsmallmatrix}
      4j'&-4k_j\\p&-4j
    \end{bsmallmatrix}  \quad\text{for }1\leq j \leq p-1,\\
\tilde{\varphi}(\gamma,\infty) &= 1.
\end{align*}
In summary, all the values we computed agree with those in Table \ref{table4}.
\end{proof}
\begin{lemma}\label{sigma composition}
    Let $p$ be an odd prime. Let $\sigma: \Lambda \times \mathbb{P}^1_{\mathbb{F}_p} \to \GL(V_p)$ be the cocycle, which is the composition of the linear representation $\pi_{\Delta}: \Delta \to \GL(V_p)$ and the cocycle $\tilde{\varphi}: \Lambda \times \mathbb{P}^1_{\mathbb{F}_p} \to \Delta$. Then the values of $\sigma$ for $\alpha,\gamma$ agree with those in Table \ref{table3}.
\end{lemma}
\begin{proof}
We will check 4 non-trivial cases.
\begin{align*}
    \pi_{\Delta} \circ \tilde{\varphi} (\alpha,p-1) &= \pi_{\Delta} \rbra{\begin{bsmallmatrix}
        1&-4\\0&1
    \end{bsmallmatrix}} = L,\\
    \pi_{\Delta} \circ \tilde{\varphi} (\alpha,\infty) &= \pi_{\Delta} \rbra{\begin{bsmallmatrix}
        1&0\\p&1
    \end{bsmallmatrix}} = 1.
\end{align*}
Since $f \in V_p$ is even, we also have
\begin{align}\label{even s(-1)}
   \pi_{\Delta} \circ \tilde{\varphi} (\gamma,0) = \pi_{\Delta} \rbra{\begin{bsmallmatrix}
        -4&0\\0&-4
    \end{bsmallmatrix}} = S_{-1} = 1.
\end{align}
For $1\leq j \leq p-1$, and $\tilde{\varphi}(\gamma,j) = \begin{bsmallmatrix}
        4j'&-4k_j\\p&-4j
    \end{bsmallmatrix}$, we have $\det \tilde{\varphi}(\gamma,j) =4 $ so that $\xi(\tilde{\varphi}(\gamma,j)) = 1$ and
$\omega(2^{-\xi(\tilde{\varphi}(\gamma,j))}\tilde{\varphi}(\gamma,j)) = \begin{bsmallmatrix}
        2j'&-2k_j\\0&-2j
    \end{bsmallmatrix}.$

Hence,
\begin{align*}
    \pi_{\Delta} \circ \tilde{\varphi} (\gamma,j) = S_{2j'}L^{-k_jj} = S_{-2j'}L^{-k_jj}
\end{align*}
as a result of (\ref{even s(-1)}). The remaining 2 cases are obvious to verify.
\end{proof}

Now we can combine all the ingredients to summarize the proof of Theorem \ref{prop theta tilde}, and the proof of Theorem \ref{prop equivariance} will follow directly.
\begin{proof}[Proof of Theorem \ref{prop theta tilde}]
Recall that for an odd prime $p$, the representation $\pi_{\Delta}: \Delta \to \GL(V_p)$ is constructed as follows: for every $g \in \Delta \subset \GL_2(\Z[\frac{1}{2}])$, we map it to an element of $\SL_2(\mathbb{F}_p)$ through the composition $\omega(2^{-\xi(g)}g)$ and then assign it to an element of $\GL(V_p)$ through the representation $\rho: P \to \GL(V_p)$ in Proposition \ref{representation P}. Meanwhile, the cocycle $\tilde{\varphi}$ is defined by the choice of section (\ref{section}) and the conjugation by $\pazocal{P}$. Finally, Lemma \ref{sigma composition} verifies that our construction is correct.
\end{proof}

Let us denote by $\pi$ the linear representation $\pi: \Lambda_{\infty} \to \GL(V_p)$ which is defined as the restriction of $\pi_{\Delta}$ to $\Lambda_{\infty}$.
\begin{rem}
Since $\Lambda$ acts transitively on $\mathbb{P}^1_{\mathbb{F}_p}$, the linear representation $\Pi : \Lambda \to \GL(V_p^X)$ defined by the action (\ref{action on vpx}) is the induced representation $\Ind^{\Lambda}_{\Lambda_{\infty}}(V_p)$ of $\Lambda$ in the sense of \cite{MR0450380}.
\end{rem}

We will use the above structure to prove that, for $d\in \N^*$ and any prime $p$ with $f \in V_p$, the weighted theta function $\theta_f$, defined in (\ref{theta z^d}), is weakly modular under some congruence subgroup.
\subsection{Proof of Theorem \ref{modular form} and \ref{cusp form}}\label{subsection proof of two theorems}
\subsubsection{Proof of Theorem \ref{modular form}, Part \RNum{1}}\label{proof of theorem part 1}

        The claim that $\theta_f$ is holomorphic on $\pazocal{H}$ follows from the fact that the $d$-dimensional theta function $\theta(\tau)=\sum_{x \in \Z^d}e^{2\pi i Q(x,x) \tau}$ is holomorphic on $\pazocal{H}$, cf. \cite[\S 4.9]{MR2112196}. Since $f$ is discrete, its norm over $(\Z/p\Z)^d$ can be bounded above by some positive constant $c_f$. It follows that the series $\theta_f$ converges absolutely and uniformly on compact subsets of $\pazocal{H}$. Since each summand is holomorphic, the conclusion follows directly.

\subsubsection{Proof of Theorem \ref{modular form}, Part \RNum{2}}\label{proof of theorem part 2}
We prove that $\theta_f$ is weakly modular of weight $\frac{d}{2}$ under some congruence subgroup.
\medskip
\paragraph{\textbf{For the even prime 2}} We have constructed in Theorem \ref{p=2 weakly modular} the representation $\pi_2: \Gamma_1(4) \to \GL_2(V_2)$ so that for all $g \in \Gamma_1(4)$ and $f \in V_2$ we have
    \begin{align*}
        [\tilde{g}]_{\frac{d}{2}} \cdot \theta_f = \theta_{\pi_2(g)f}.
    \end{align*}
We claim that the group $\Gamma_2$ defined by (\ref{group gamma2}) satisfies
\begin{align*}
    \Gamma_2 = \ker \pi_2 = \bbra{ g \in \Gamma_1(4): \pi_2(g)=1}.
\end{align*}
Indeed, since $\pi_2 = \rho_2 \circ \omega_2 \circ \mu$, to describe $\ker\pi_2$ is to find those $g = \begin{bsmallmatrix}
    a&b\\4c&d
\end{bsmallmatrix} \in \Gamma_1(4)$ for which $\rho_2 \circ \omega_2(\gamma g \gamma^{-1}) = 1$. For such $g$,
\begin{align*}
    \gamma g \gamma^{-1} = \begin{bmatrix}
        d &-c\\ -4b&a
    \end{bmatrix}  = \begin{bmatrix}
        1 & 0\\ 0&1
    \end{bmatrix} \in \SL_2(\Z/4\Z)
\end{align*}
and it follows that $c=0[4]$. Finally,
\begin{align*}
\ker\pi_2 = \bbra{ g = \begin{bmatrix}
        a&b\\c&d
    \end{bmatrix} \in \SL_2(\Z): a,d=1[4],c=0[16]},
\end{align*}
as required. In particular, for the even prime 2, any $g \in \Gamma_2$ and $f \in V_2$, we have
\begin{align*}
    [\tilde{g}]_{\frac{d}{2}} \cdot \theta_f = \theta_f,
\end{align*}
which means $\theta_f$ is weakly modular of weight $\frac{d}{2}$ under $\Gamma_2$.

\medskip
         \paragraph{\textbf{For any odd prime $p$}} We have constructed in Theorem \ref{prop theta tilde} the representation $\pi_{\Delta}: \Delta \to \GL(V_p)$ and the cocycle $\tilde{\varphi}: \Lambda \times \mathbb{P}^1_{\mathbb{F}_p} \to \Delta$ which defines the cocycle $\sigma: \Lambda \times \mathbb{P}^1_{\mathbb{F}_p} \to \GL(V_p)$ so that for $u \in \C$ to be a solution of the equation $\rbra{\frac{i}{2}}^{\frac{1}{2}} = \rbra{\frac{1}{4}}^u$, for any $g \in \Lambda$, $j \in \mathbb{P}^1_{\mathbb{F}_p}$ and $f \in V_p$ the following holds
        \begin{align*}
            [\tilde{g}]_{\frac{d}{2},u} \cdot \theta^j_{f} = \theta^{g\cdot j}_{\sigma(g,j)f} = \theta^{g\cdot j}_{\pi_{\Delta} \circ \tilde{\varphi}(g,j)f}.
        \end{align*}
        In particular, for $j = \infty$ and any $g \in \Lambda_{\infty}$ we get
        \begin{align*}
            [\tilde{g}]_{\frac{d}{2},u} \cdot \theta^{\infty}_{f} = \theta^{\infty}_{\pi_{\Delta} \circ \tilde{\varphi}(g,\infty)f}.
        \end{align*}
        By definition we have $\theta_f = p^{-d} \theta^\infty_{\mathcal{F}(f)}$. By our choice of section (\ref{section}) at $\infty$, we know for $g \in \Lambda_{\infty}$ that $\varphi(g,\infty) = g$, hence $\pi_{\Delta} \circ \tilde{\varphi}(g,\infty) = \pi_{\Delta}(\pazocal{P}g\pazocal{P}^{-1})$. This implies
        \begin{align}\label{theta f weakly modular}
            [\tilde{g}]_{\frac{d}{2},u} \cdot \theta_{f} = \theta_{\pi_{\Delta}(\pazocal{P}g\pazocal{P}^{-1})f}.
        \end{align}
         We claim that the group $\Gamma_p$ defined by (\ref{group gamma_p}) satisfies $$\Gamma_p = \pazocal{P}^{-1}\ker \pi_{\Delta}\pazocal{P} \cap \SL_2(\Z)= \bbra{ g \in \Lambda_{\infty} \cap \SL_2(\Z): \pi_{\Delta}(\pazocal{P}g\pazocal{P}^{-1}) = 1}.$$
        By Lemma \ref{lemma Lambda inside SL2Z} and \ref{lemma description of lambda infty}, we know that $\Lambda_{\infty} \cap \SL_2(\Z) = \Gamma_1(4) \cap \Delta$ so that
        \begin{align*}
        \Lambda_{\infty} \cap \SL_2(\Z) = \bbra{ g = \begin{bmatrix}
        a&b\\c&d
    \end{bmatrix} \in \SL_2(\Z): a,d=1[4],c=0[4p]}.
        \end{align*}
    by the Chinese reminder theorem. Therefore, to describe $\pazocal{P}^{-1}\ker \pi_{\Delta}\pazocal{P} \cap \SL_2(\Z)$ is to find those $g = \begin{bsmallmatrix}
    a&b\\4pc&d
\end{bsmallmatrix} \in \Lambda_{\infty} \cap \SL_2(\Z)$ for which $\pi_{\Delta}(\pazocal{P}g\pazocal{P}^{-1}) =1$. For such $g$,
\begin{align*}
    \pazocal{P}g\pazocal{P}^{-1} = \begin{bmatrix}
        d&4c\\bp&a
    \end{bmatrix} = \begin{bmatrix}
        1&0\\0&1
    \end{bmatrix} \in \SL_2(\mathbb{F}_p)
\end{align*}
and it follows that $c=0[p]$, $a,d=1[p]$. Finally,
\begin{align*}
    \pazocal{P}^{-1}\ker \pi_{\Delta}\pazocal{P} \cap \SL_2(\Z) = \bbra{ g = \begin{bmatrix}
        a&b\\c&d
    \end{bmatrix} \in \SL_2(\Z): a,d=1[4p],c=0[4p^2]},
\end{align*}
as required. In particular, for any odd prime $p$, for $g \in \Gamma_p$ and $f \in V_p$, by (\ref{theta f weakly modular}) we have
\begin{align*}
    [\tilde{g}]_{\frac{d}{2}} \cdot \theta_{f} = \theta_{\pi_{\Delta}(\pazocal{P}g\pazocal{P}^{-1})f} = \theta_f,
\end{align*}
which means $\theta_{f}$ is weakly modular of weight $\frac{d}{2}$ under $\Gamma_p$.

\subsubsection{Proof of Theorem \ref{modular form}, Part \RNum{3}}
We prove that $\theta_f$ is holomorphic at all $s\in \mathbb{P}^1_{\Q}$. By definition (\ref{theta fourier series}) we know that $\theta_f$ is holomorphic at $\infty$. For $s \in \Q$, let $u_s := \begin{bsmallmatrix}
                1&-s\\0&1
            \end{bsmallmatrix}$. We set
        \begin{align*}
            h_s := \gamma u_s,
        \end{align*}
which is a matrix in $\GL_2^+(\Q)$ with $h_s \cdot s = \infty$. Let $u$ be a solution of the equation $\rbra{\frac{i}{2}}^{\frac{d}{2}} = \rbra{\frac{1}{4}}^u$. We claim that for any $s \in \Q$ and let $\dot{h_s} \in \pazocal{G}$, we have a Fourier series expansion $\rbra{[\dot{h_s}]_{\frac{d}{2},u} \cdot \theta_f}(\tau) = \sum_{n=0}^\infty a_n e^{2\pi i \tau n/m}$ for some $m \in \N^*$ and $a_n \in \C$.

\begin{prop}\label{prop theta_f holomophic at all s}
Let $d \in \N^*$, $p$ be a prime and $f \in V_p$. For every $s \in \Q$ in any of its quotient forms $\frac{w}{v}$ with $w \in \Z$ and $v \in \N^*$, we decompose the denominator $v$ as $v = cp^r$ for some $c \in \N^*$ and $r \in \N$ with $\gcd(c,p)=1$ and we let $\tilde{r} = \max\{r,1\}$. With $u$ defined above, for every $s \in \Q$, there exists some $F : (\Z/cp^{\tilde{r}}\Z)^d \to \C$ defined as
\begin{align}\label{fourier transform of A}
    F(\xi) := \sum_{x \in (\Z/cp^{\tilde{r}}\Z)^d} f(x)e^{2\pi i Q(x,x)\frac{w}{cp^r}} e^{-2\pi i Q(x,\xi)\frac{1}{cp^{\tilde{r}}}}
\end{align}
such that for some $\varepsilon \in \bbra{\pm 1}$ the following holds
\begin{equation}\label{holomorphic sum in lattice}
   \rbra{[\dot{h_s}]_{\frac{d}{2},u} \cdot \theta_f }(\tau)= \frac{1}{(\varepsilon cp^{\tilde{r}})^d} \sum_{\xi \in \Z^d} F(\xi[cp^{\tilde{r}}]) e^{2\pi i Q(\xi,\xi)\frac{1}{(cp^{\tilde{r}})^2} \tau},\quad \tau \in \pazocal{H}.
\end{equation}
\end{prop}
\begin{proof}
We know that for some $\varepsilon_1 \in \bbra{\pm 1}$
\begin{align*}
    \rbra{[\dot{u_s}]_{\frac{d}{2}} \cdot \theta_f} (\tau) = \varepsilon_1^{-d}\theta_f (\tau+s) &= \varepsilon_1^{-d}\sum_{x \in \Z^d} f(x[p]) e^{2\pi i Q(x,x)(\tau+s)}\\
    &=\varepsilon_1^{-d}\sum_{x \in \Z^d} f(x[p]) e^{2\pi i Q(x,x)\frac{w}{cp^r}}e^{2\pi i Q(x,x)\tau}, \quad \tau \in \pazocal{H}.
\end{align*}
Since $f \in V_p$, the function $A(x) := f(x)e^{2\pi i Q(x,x)\frac{w}{cp^r}}$ is defined on $(\Z/cp^{\tilde{r}}\Z)^d$. Recall we have defined $\eta_{\tau}(t) = e^{2\pi i Q(t,t)\tau}$ for $t\in \R^d$. Fix any $\tau \in \pazocal{H}$, the function
\begin{align*}
    \phi(s,t) := A(s) \eta_{\tau}(t), \quad (s,t) \in (\Z/cp^{\tilde{r}}\Z)^d \times \R^d
\end{align*}
in $(s,t)$ is defined on $(\Z/cp^{\tilde{r}}\Z)^d \times \R^d$, and it follows that $\phi \in S((\Z/cp^{\tilde{r}}\Z)^d \times \R^d)$. By definition, $F(\xi) = \mathcal{F}_{(\Z/cp^{\tilde{r}}\Z)^d}(A)(\xi)$ is the Fourier transform of the function $A$, then by Corollary \ref{corollary poisson} and (\ref{fourier transform gaussian}), we have
\begin{align*}
    \rbra{[\dot{u_s}]_{\frac{d}{2}} \cdot \theta_f }(\tau) &= \varepsilon_1^{-d}\sum_{t \in \Z^d} \phi(t[cp^{\tilde{r}}],t) = \frac{1}{(\varepsilon_1 cp^{\tilde{r}})^d} \sum_{\xi \in \Z^d} \mathcal{F}(\phi)(-\xi[cp^{\tilde{r}}],(cp^{\tilde{r}})^{-1}\xi)\\
    &= \rbra{\frac{i}{2\tau}}^{\frac{d}{2}} \frac{1}{(\varepsilon_1 cp^{\tilde{r}})^d} \sum_{\xi \in \Z^d} F(\xi[cp^{\tilde{r}}]) e^{2\pi iQ(\xi,\xi)\frac{1}{(cp^{\tilde{r}})^2} \rbra{-\frac{1}{4\tau}}}.
\end{align*}
By sending $\tau \mapsto -\frac{1}{4\tau}$ in the above equality we have
\begin{align*}
     \rbra{[\dot{u_s}]_{\frac{d}{2}} \cdot \theta_f} \rbra{-\frac{1}{4\tau}} = \rbra{\frac{2\tau}{i}}^{\frac{d}{2}} \frac{1}{(\varepsilon_1 cp^{\tilde{r}})^d} \sum_{\xi \in \Z^d} F(\xi[cp^{\tilde{r}}]) e^{2\pi i Q(\xi,\xi)\frac{1}{(cp^{\tilde{r}})^2} \tau}.
\end{align*}
Since $u$ is a solution of the equation $\rbra{\frac{i}{2}}^{\frac{d}{2}} = \rbra{\frac{1}{4}}^u$, it follows that for some $\varepsilon_2 \in \bbra{\pm 1}$
\begin{align*}
    \rbra{[\dot{h_s}]_{\frac{d}{2},u} \cdot \theta_f }(\tau)&= \rbra{[\dot{\gamma}]_{\frac{d}{2},u} \cdot [\dot{u_s}]_{\frac{d}{2}} \cdot \theta_f}(\tau)= \varepsilon_2^{-d}\tau^{-\frac{d}{2}}4^{-u} \rbra{[\dot{u_s}]_{\frac{d}{2}} \cdot \theta_f }(\gamma^{-1}\cdot \tau)  \\
    &= \frac{1}{(\varepsilon_1 \varepsilon_2 cp^{\tilde{r}})^d} \sum_{\xi \in \Z^d} F(\xi[cp^{\tilde{r}}]) e^{2\pi i Q(\xi,\xi)\frac{1}{(cp^{\tilde{r}})^2} \tau}.
\end{align*}
Let $\varepsilon := \varepsilon_1 \varepsilon_2$, the conclusion follows.
\end{proof}
Now (\ref{holomorphic sum in lattice}) can be rewritten as follows
\begin{align*}
    \rbra{[\dot{h_s}]_{\frac{d}{2},u} \cdot \theta_f }(\tau) = \frac{1}{(\varepsilon cp^{\tilde{r}})^d} \sum_{n=0}^{\infty} \rbra{\sum_{x \in X_d(n)}F(x[cp^{\tilde{r}}])} q^n_{(cp^{\tilde{r}})^2}, \quad q_{(cp^{\tilde{r}})^2} = e^{2\pi i \tau / (cp^{\tilde{r}})^2}, \tau \in \pazocal{H}.
\end{align*}
So for any $s \in \Q$, the function $[\dot{h_s}]_{\frac{d}{2},u} \cdot \theta_f$ is indeed in the form of a Fourier expansion with no negative powers of $q_{(cp^{\tilde{r}})^2}$, and we have proved that $\theta_f$ is holomorphic at all $s \in \mathbb{P}_{\Q}^1$ by Definition \ref{defi holomorphic at cusps half integral}.
\begin{rem}
     Let $u$ be defined as before. In fact one can prove directly that for all $h\in \GL_2^+(\Q)$ the weakly modular function $[\dot{h}]_{\frac{d}{2},u} \cdot \theta_{f}$ is holomorphic at $\infty$ by \cite[Proposition 1.2.4]{MR2112196}. Since we will later use (\ref{fourier transform of A}) to establish our cusp form criteria, the proposition above is also interesting in its own right.
\end{rem}

\subsubsection{Proof of Theorem \ref{cusp form}}
Let $d \in \N^*$, $p$ be a prime and $f \in V_p$. For any $r \in \N$, $w \in \Z$ and $\tilde{r}= \max\{r,1\}$, we define
\begin{align}
    S(r,w) := \sum_{y \in (\Z/p^{\tilde{r}}\Z)^d} f(y)e^{2\pi i Q(y,y)\frac{w}{p^r}},
\end{align}
where $f(y)$ means that $f$ is evaluated at the image of $y$ in $(\Z/p\Z)^d$. The following lemma will help us to prove Theorem \ref{cusp form} in both directions.
\begin{lemma}\label{lemma bezout}
Let $d \in \N^*$, $p$ be a prime and $f \in V_p$. For any $c \in \N^*$ with $\gcd(c,p)=1$, $r \in \N$, $w \in \Z$ and $\tilde{r}= \max\{r,1\}$, there exist $m_1,m_2 \in \Z$ such that the following decomposition holds
\begin{align}\label{decomposition of cusp form}
    \sum_{x \in (\Z/cp^{\tilde{r}}\Z)^d} f(x)e^{2\pi i Q(x,x)\frac{w}{cp^r}} = S(r,m_1w)\sum_{z \in (\Z/c\Z)^d} e^{2\pi i Q(z,z)\frac{m_2w}{c}}.
\end{align}
\end{lemma}
\begin{proof}
    Since $\gcd(c,p^r) = 1$ by assumption, Bézout's identity implies that there exist $m_1,m_2 \in \Z$ such that $m_1c + m_2 p^r = 1$. Hence, one can write
    \begin{align*}
        \frac{w}{cp^r} =  \frac{w}{cp^r}(m_1c + m_2 p^r) = \frac{m_1w}{p^r} + \frac{m_2 w}{c}.
    \end{align*}
    The result follows naturally.
\end{proof}
(1) We first show that for $d \in \N^*$ and any $f:(\Z/p\Z)^d\to \C$, if the weighted theta function $\theta_f$ is a cusp form, then (\ref{vanishing condition p=2}) and (\ref{vanishing condition}) hold respectively for the even prime 2 and any odd prime $p$. By definition (\ref{theta fourier series}) we know that since $\theta_f$ is zero at $\infty$, we necessarily have $\sum_{x \in X_d(0)}f(x[p]) = 0$, which means $f(0,\cdots,0)=0$. For $s \in \Q$, recall that we have defined $h_s = \gamma u_s$ as in the previous subsection so that $h_s \cdot s = \infty$. Since $\theta_f$ is zero at $s \in \Q$, by Lemma \ref{lemma holomorphic at s existence} we know that the function $[\dot{h_s}]_{\frac{d}{2},u} \cdot \theta_f$ is zero at $\infty$ for $\dot{h_s} \in \pazocal{G}$. This means that for all $s$ as we decompose it in Proposition \ref{prop theta_f holomophic at all s}, we have $F(0)=0$ in (\ref{fourier transform of A}), i.e. by (\ref{decomposition of cusp form})
\begin{align*}
    \sum_{x \in (\Z/cp^{\tilde{r}}\Z)^d} f(x)e^{2\pi i Q(x,x)\frac{w}{cp^r}} = S(r,m_1w)\sum_{z \in (\Z/c\Z)^d} e^{2\pi i Q(z,z)\frac{m_2w}{c}} = 0
\end{align*}
for all $r\in \N$, $m_1,w \in \Z$. In particular, after taking $c=1$ above, for any $r \geq 1$ and $w \in \Z$ we know that $S(r,w)=0$.

\medskip
\paragraph{\textbf{For the even prime 2}} Let $w \in \Z$. Take $r=2$ and we have
\begin{align}
    S(2,w) &= \sum_{x \in (\Z/4\Z)^d} f(x)e^{2\pi i Q(x,x)\frac{w}{4}} = 2^d \sum_{x \in (\Z/2\Z)^d} f(x)e^{2\pi i Q(x,x)\frac{w}{4}}\nonumber\\
    &=2^d \sum_{a \in \Z/4\Z} \bbra{ \sum_{\substack{x \in (\Z/2\Z)^d\\ Q(x,x)=a[4]}} f(x) } e^{2\pi i \frac{aw}{4}}.\label{p=2s2}
\end{align}
Hence imposing $S(2,w)=0$ for all $w \in \Z$, we necessarily have $\forall a \in \Z/4\Z: \sum_{x \in X_{2,d}(a)}f(x) =0$.

For $r \geq 3$, $k \in \{0,1\}^d$ and $w \in \Z$, we define
\begin{align}
    R(r,k,w) := \sum_{u \in (\Z/2^{r-2}\Z)^d} e^{2\pi i Q(u,u+k) \frac{w}{2^{r-2}}}.
\end{align}
The following lemma simplifies the calculations of $S(r,w)$ for $r \geq 3$.
\begin{lemma}\label{lemma R3}
For $r \geq 3$ , $k \in \{0,1\}^d$ and $w \in \Z$, we have
\begin{align}
    R(r,k,w) = 0 \begin{cases}
        \text{ for }r = 3 \text{ and } k \in \{0,1\}^d \setminus \{1,\cdots,1\}\\
        \text{ for }r \geq 4\text{ and }k \in \{0,1\}^d \setminus \{0,\cdots,0\}
        \end{cases}.
\end{align}
\end{lemma}
\begin{proof}
For $r=3$, $k \in \{0,1\}$ and $u \in \Z/2\Z$, let us call $h_k: \Z/2\Z \to \Z/2\Z$ defined by $h_k (u):= u^2 + ku$, then it follows that for every $u \in \Z/2\Z$ we have $h_0(u) = u^2 \in \{0,1\}$ and $h_1(u) = u^2 + u =0$. Hence, we know that $e^{2 \pi i \frac{h_0(0)}{2}} = -e^{2 \pi i \frac{h_0(1)}{2}}$ and $e^{2 \pi i \frac{h_1(u)}{2}} = 1 \quad\text{for }u \in \Z/2\Z$. In general, for $k \in \{0,1\}^d$ and $w \in \Z$ it follows that
\begin{align*}
    R(3,\{1,\cdots,1\},w) &= \sum_{u \in (\Z/2\Z)^d} \prod_{i=1}^d e^{2\pi i \frac{h_1(u_i)}{2} w} = 2^d,\\
    R(3,k,w) &= 0,\quad k \in \{0,1\}^d \setminus \{1,\cdots,1\}
\end{align*}
as desired.

For $r \geq 4$, $k \in \{0,1\}$ and $u \in \Z/2^{r-2}\Z$, let us call $g_k: \Z/2^{r-2}\Z \to \Z/2^{r-2}\Z$ defined by $g_k(u):=u^2+ku$, then it follows that for every $u \in \Z/2^{r-2}\Z$ we have $g_0(u) = g_0(u+2^{r-3}) = u^2$ and $g_1(u) +2^{r-3}=g_1(u+2^{r-3}) = u^2+u+2^{r-3}$. Hence for every $u \in \Z/2^{r-2}\Z$ we get $e^{2\pi i \frac{g_0(u)}{2^{r-2}}} = e^{2\pi i \frac{g_0(u+2^{r-3})}{2^{r-2}}}$ and $e^{2\pi i \frac{g_1(u)}{2^{r-2}}} =-e^{2\pi i \frac{g_1(u+2^{r-3})}{2^{r-2}}}$, and in general for $k \in \{0,1\}^d$ and $w \in \Z$ it follows that
\begin{align*}
    R(r,\{0,\cdots,0\},w) &= \sum_{u \in (\Z/2^{r-2}\Z)^d}\prod_{i=1}^d e^{2\pi i \frac{g_0(u_i)}{2^{r-2}}w} =\sum_{u \in (\Z/2^{r-2}\Z)^d}\prod_{i=1}^d e^{2\pi i \frac{u_i^2}{2^{r-2}}w},\\
    R(r,k,w) &= 0, \quad k \in \{0,1\}^d \setminus \{0,\cdots,0\}
\end{align*}
as desired.
\end{proof}
Now we calculate $S(3,w)$ for $w \in \Z$ as follows
\begin{align}
    S(3,w) &= \sum_{x \in (\Z/8\Z)^d} f(x)e^{2\pi i Q(x,x)\frac{w}{8}} = 2^d \sum_{x \in (\Z/4\Z)^d} f(x)e^{2\pi i Q(x,x)\frac{w}{8}} \nonumber\\
    &= 2^d \sum_{k \in \{0,1\}^d} f(k) e^{2\pi i Q(k,k) \frac{w}{8}} \sum_{u \in (\Z/2\Z)^d} e^{2\pi i Q(u,u+k) \frac{w}{2}}\nonumber\\
    &= 2^d \sum_{k \in \{0,1\}^d} f(k) e^{2\pi i Q(k,k) \frac{w}{8}} R(3,k,w) = 4^d f(1,\cdots,1)e^{2\pi i Q(1,1) \frac{w}{8}}  \label{p=2 s3},
\end{align}
by Lemma \ref{lemma R3}. Hence the condition $S(3,w)=0$ for all $w \in \Z$ means that we necessarily have $f(1,\cdots,1)=0$, and in summary (\ref{vanishing condition p=2}) holds by Proposition \ref{prop theta_f holomophic at all s} and Lemma \ref{lemma bezout}.
\medskip
\paragraph{\textbf{For any odd prime $p$}} Let $w \in \Z$, take $r=1$ we have
\begin{align}
    S(1,w) = \sum_{x \in (\Z/p\Z)^d} f(x)e^{2\pi i Q(x,x)\frac{w}{p}} = \sum_{a \in \Z/p\Z} \bbra{ \sum_{\substack{x \in (\Z/p\Z)^d\\ Q(x,x)=a}} f(x) } e^{2\pi i \frac{aw}{p}}\label{p prime s1}.
\end{align}
Hence the condition $S(1,w)=0$ for all $w \in \Z$ means we necessarily have $\forall a \in \Z/p\Z: \sum_{x \in X_{p,d}(a)}f(x) =0$. In sum, (\ref{vanishing condition}) holds by Proposition \ref{prop theta_f holomophic at all s} and Lemma \ref{lemma bezout}.

\medskip
(2) For the converse direction, we show that for $d \in \N^*$ and any $f:(\Z/p\Z)^d \to \C$, if (\ref{vanishing condition p=2}) and (\ref{vanishing condition}) hold respectively for the even prime 2 and any odd prime $p$, then the weighted theta function $\theta_f$ is a cusp form.
\medskip
\paragraph{\textbf{For the even prime 2}}
Now suppose that (\ref{vanishing condition p=2}) holds, then $f(0,\cdots,0) = 0$ implies that $\theta_f$ is zero at $\infty$ by definition (\ref{theta fourier series}). In addition, let $w \in \Z$. For $r = 0,1$ we have
\begin{align*}
S(0,w) &=\sum_{x \in (\Z/2\Z)^d} f(x),\\
    S(1,w) &= \sum_{x \in (\Z/2\Z)^d} f(x)e^{2\pi i Q(x,x)\frac{w}{2}} = \sum_{a \in \Z/2\Z} \bbra{ \sum_{\substack{x \in (\Z/2\Z)^d\\ Q(x,x)=a}} f(x) } e^{2\pi i \frac{aw}{2}},
\end{align*}
so by assumption, (\ref{p=2s2}) and (\ref{p=2 s3}), we have $S(r,w)=0$ for $r=0,1,2,3$. For $r \geq 4$, we have
\begin{align*}
    S(r,w) &= \sum_{x \in (\Z/2^r\Z)^d} f(x)e^{2\pi i Q(x,x)\frac{w}{2^r}} = 2^d \sum_{x \in (\Z/2^{r-1}\Z)^d} f(x)e^{2\pi i Q(x,x)\frac{w}{2^r}}\\
    &= 2^d\sum_{k \in \{0,1\}^d}f(k) e^{2\pi i Q(k,k)\frac{w}{2^r}} \sum_{u \in (\Z/2^{r-2}\Z)^d} e^{2\pi i Q(u,u+k)\frac{w}{2^{r-2}}}\\
    &= 2^d\sum_{k \in \{0,1\}^d}f(k) e^{2\pi i Q(k,k)\frac{w}{2^r}} R(r,k,w) = 2^d f(0,\cdots,0) R(r,\{0,\cdots,0\},w) = 0
\end{align*}
by Lemma \ref{lemma R3} and our assumption. This means that for all $s\in \Q$ as we decompose it in Proposition \ref{prop theta_f holomophic at all s}, we have $F(0)=0$ in (\ref{fourier transform of A}). Since for every $s \in \Q$, the 0-th Fourier coefficient $F(0)$ of the weakly modular form $[\dot{h_s}]_{\frac{d}{2},u} \cdot \theta_f$ vanishes, we conclude that $\theta_f$ is zero at all the cusps by Definition \ref{defi holomorphic at cusps half integral}.
\medskip
\paragraph{\textbf{For any odd prime $p$}} For $r \geq 2$, $k \in (\Z/p^r \Z)^d$ and $w \in \Z$, we define
\begin{align}
    T(r,k,w) := \sum_{u \in (\Z/p^{r-1}\Z)^d} e^{2\pi i Q(k+pu,k+pu) \frac{w}{p^r}}.
\end{align}
The following lemma simplifies the calculations of $S(r,w)$ for $r \geq 2$.
\begin{lemma}\label{lemma T(r,k,w)}
Let $p$ be an odd prime. For $r \geq 2$, $k \in (\Z/p^r \Z)^d$ and $w \in \Z$, we have
\begin{align}
    T(r,k,w) = \begin{cases}
        p^d\sum_{v \in (\Z/ p^{r-2}\Z)^d} e^{2\pi i Q(v,v) \frac{w}{p^{r-2}}}, &\quad k \in (p\Z/p^r\Z)^d\\
        0, &\quad k \notin (p\Z/p^r\Z)^d
        \end{cases}.
\end{align}
\end{lemma}
\begin{proof}
    For $r \geq 2$, $k \in (\Z/p^r \Z)^d$ and $w \in \Z$,
    \begin{align*}
        T(r,k,w) = e^{2\pi i  \frac{Q(k,k)}{p^r} w} \sum_{u \in (\Z/p^{r-1}\Z)^d} e^{2\pi i \sbra{\frac{2Q(k,u)}{p^{r-1}}+\frac{Q(u,u)}{p^{r-2}}}w}.
    \end{align*}
    Fix $v \in (\Z/p\Z)^d$, then the sum in $T(r,k,w)$ is invariant under the shift $u \mapsto u + p^{r-2}v$, i.e.
    \begin{align*}
        T(r,k,w) = e^{2\pi i  \frac{Q(k,k)}{p^r}w} \sum_{u \in (\Z/p^{r-1}\Z)^d} e^{2\pi i \sbra{\frac{2Q(k,u+p^{r-2}v)}{p^{r-1}}+\frac{Q(u+p^{r-2}v,u+p^{r-2}v)}{p^{r-2}}}w} = e^{2\pi i  \frac{2Q(k,v)}{p}w} T(r,k,w).
    \end{align*}
    Hence, if $k \notin (p\Z/p^r\Z)^d$, then we have $T(r,k,w) = 0$. If $k \in (p\Z/p^r\Z)^d$, then there exists $l \in (\Z/p^{r-1}\Z)^d$ such that $k = pl$ and
    \begin{align*}
        T(r,k,w) = \sum_{u \in (\Z/p^{r-1}\Z)^d} e^{2\pi i Q(l+u,l+u) \frac{w}{p^{r-2}}} = \sum_{u \in (\Z/p^{r-1}\Z)^d} e^{2\pi i Q(u,u) \frac{w}{p^{r-2}}} = p^d\sum_{v \in (\Z/ p^{r-2}\Z)^d} e^{2\pi i Q(v,v) \frac{w}{p^{r-2}}}.
    \end{align*}
    Hence, the statement follows.
\end{proof}

Now suppose that (\ref{vanishing condition}) holds, then $f(0,\cdots,0) = 0$ implies that $\theta_f$ is zero at $\infty$ by definition (\ref{theta fourier series}). In addition, let $w \in \Z$. For $r = 0$ we have
\begin{align*}
    S(0,w) &=\sum_{x \in (\Z/p\Z)^d} f(x),
\end{align*}
so by assumption and (\ref{p prime s1}), we have $S(0,w)=S(1,w)=0$. For $r \geq 2$ we have
\begin{align*}
    S(r,w) &= \sum_{x \in (\Z/p^r\Z)^d} f(x)e^{2\pi i Q(x,x)\frac{w}{p^r}} = \sum_{k \in \{0,\cdots,p-1\}^d} f(k) \sum_{u \in (\Z/p^{r-1}\Z)^d}e^{2\pi i Q(k+pu,k+pu)\frac{w}{p^r}} \\
    &= \sum_{k \in \{0,\cdots,p-1\}^d} f(k) T(r,k,w) = f(0,\cdots,0)p^d \sum_{v \in (\Z/ p^{r-2}\Z)^d} e^{2\pi i Q(v,v) \frac{w}{p^{r-2}}} = 0
\end{align*}
by Lemma \ref{lemma T(r,k,w)} and our assumption. This means that for all $s\in \Q$ as we decompose it in Proposition \ref{prop theta_f holomophic at all s}, we have $F(0)=0$ in (\ref{fourier transform of A}). Since for every $s \in \Q$, the 0-th Fourier coefficient $F(0)$ of the weakly modular form $[\dot{h_s}]_{\frac{d}{2},u} \cdot \theta_f$ vanishes, we conclude that $\theta_f$ is zero at all the cusps by Definition \ref{defi holomorphic at cusps half integral}.

\section{Local equidistribution results}\label{section 3}
In this section, we prove Theorem \ref{theorem equidistribution}. Recall in \S \ref{section 2}, we have proved in Theorem \ref{modular form} that the weighted theta function $\theta_f$ defined by (\ref{theta z^d}) is a modular form of weight $\frac{d}{2}$, and a necessary and sufficient condition for it to be a cusp form in Theorem \ref{cusp form}. These preparations allow us to use bounds for the Fourier coefficients of cusp forms, which are introduced in \S \ref{section 3.1}. Next, we quantify the representation number by the standard quadratic form in \S \ref{section 3.2}. For dimension 4, an explicit formula by Jacobi is applicable already. For higher dimensions larger than 4, this number is characterized by a singular series, where the description is sufficient for us to prove Theorem \ref{theorem equidistribution} on non-zero level sets, i.e. on $X_{p,d}(a)$ where $a \neq 0$. On $X_{p,d}(0) \setminus \{0,\cdots,0\}$, we need a further precision of the singular series. To do this, we have explicitly computed the $p$-adic local density for each odd $p$ in Corollary \ref{cor explicit formula for p-adic density}. This calculation finally leads to Proposition \ref{prop difference bound}, which allows us to conclude the equidistribution phenomenon on each level set.
\subsection*{A few notes on asymptotic notations}
We recall the \textit{Bachmann-Landau} notation. Let $f: \R \to \C$ and $g: \R \to \R_{> 0}$. We write $f(x) = O(g(x))$ as $x \to \infty$ as meaning that $\limsup_{x \to \infty} \abs{f(x)}/g(x)<\infty$. We write $f(x)=o(g(x))$ when $\lim_{x \to \infty} \abs{f(x)}/g(x) = 0$. We also use \textit{Vinogradov's} notation as follows. In the scenario just introduced, we write $f(x) \ll g(x)$ to mean that $f(x) = O(g(x))$. Also, we write $f(x) \gg g(x)$ when both $f$ and $g$ are non-negative and $g(x) \ll f(x)$. Finally we write $f(x) \asymp g(x)$ when one has both $f(x) \ll g(x)$ and $g(x) \ll f(x)$.
\subsection{Bounds for the Fourier coefficients of cusp forms}\label{section 3.1}
Let $\Gamma$ be a congruence subgroup of $\Gamma_1(4)$ and let
\begin{align}\label{fourier expansion of cusp forms}
\psi(\tau) = \sum_{n=1}^{\infty}a_nq^n,\quad q = e^{2\pi i \tau},\tau \in \pazocal{H}
\end{align}
be a cusp form of weight $k \in \frac{1}{2}\Z$ under $\Gamma$.

The well-known Hecke bound (cf. \cite[(5.7)]{MR1474964}), which states that $a_n \ll n^{k/2}$, is already strong enough for our application when $d \geq 5$. However, an improvement is required for our equidistribution theory when $d =4$, which can be achieved using a non-trivial estimate for the Kloosterman sums that appear in the Fourier expansion of Poincaré series. The following statement is a result from \cite[(5.19)]{MR1474964}.
\begin{thm}[Kloosterman sums bound]\label{thm bound kloosterman sums}
Let $\psi$ be a cusp form of weight $2$ under $\Gamma$, which admits the Fourier expansion (\ref{fourier expansion of cusp forms}), then for any $\varepsilon>0$ its $n$-th Fourier coefficients satisfy
\begin{equation}
    a_n \ll n^{3/4 + \varepsilon}.
\end{equation}
\end{thm}
\begin{rem}\label{rem kloosterman sums}
For a cusp form of weight $k$, where $k$ is half an odd integer such as $3/2$, one can obtain the bound $a_n \ll n^{k/2 - 1/4 + \varepsilon}$ for any $\varepsilon>0$. The proof relies on an estimate for the Kloosterman sums twisted by the extended quadratic symbol, known as the Salié sum. For dimension 3, an even sharper bound, if one exists, is needed for application (cf. \S \ref{section 4}).
\end{rem}
\subsection{The representation number by quadratic forms}\label{section 3.2}
In this subsection, we present a series of results on the representation number $r_d(n)$ of the standard quadratic form $Q$. For low dimensions, such as $d=4$, these results are classically known.
\begin{thm}[cf. {\cite[\S 11.3]{MR1474964}}]\label{thm rn for d =3 or 4}
    Let $n \in \N$. The representation number by the standard quadratic form $Q$ in dimension 4 satisfies
       \begin{align}
        r_4(n) = 8(2+(-1)^n) \sum_{d|n,2 \nmid d} d.
    \end{align}
\end{thm}
\begin{rem}\label{rem r4}
    Note that by the exact formula above for $r_4(n)$, when $n = 2^k$ for $k \in \N$, the representation number is $r_4(2^k) = 24$, which does not provide enough candidates for equidistribution.
\end{rem}

For $d \geq 5$, effective estimates for $r_d(n)$ are provided by the circle method. We first present a characterization of $r_d(n)$ in terms of a singular series, which was adapted by Kloosterman (cf. \cite[\S 11]{MR1474964}). The original statement applies to any positive-definite quadratic form. In particular, the following case is relevant to our application.

\begin{thm}[cf. {\cite[Theorem 11.2]{MR1474964}}]\label{thm representation number}
    Let $d \geq 5$ and $n \in \N$. The representation number by the standard quadratic form $Q$ satisfies
    \begin{equation}
        r_d(n) = \frac{\pi^{d/2}}{\Gamma(d/2)}n^{d/2 -1}\mathfrak{S}_d(n) + o(n^{d/2-1}),
    \end{equation}
    where the singular series $\mathfrak{S}_d(n)$ is defined by
    \begin{equation}\label{eq Ad}
        \mathfrak{S}_d(n) = \sum_{q=1}^{\infty}A_d(q,n), \quad A_d(q,n) = \sum_{\substack{a=1\\(a,q)=1}}^{q}  \rbra{q^{-1} S(q,a)}^d e^{2\pi i \rbra{\frac{-na}{q}}}.
    \end{equation}
    Note that by convention $A_d(1,n)=1$. Here, $S(q,a)$ denotes the Gauss sum, which is defined by
    \begin{equation}\label{gauss sum formula}
        S(q,a) = \sum_{t=1}^q e^{2 \pi i \rbra{\frac{at^2}{q}}}.
    \end{equation}
    Also, $\Gamma(z)$ denotes the familiar $\Gamma$-function defined by $\Gamma(z) = \int_0^\infty t^{z-1} e^{-t} dt \quad(\re(z)>0)$.
\end{thm}

To derive useful bounds for $r_d(n)$ for our later application, we introduce the notion of $p$-adic local density (cf. \cite[\S 7]{circle}) by 
\begin{equation}\label{defi p-adic local density}
  \delta_{p,d}(n) = \sum_{h=0}^{\infty}A_d(p^h,n), \quad \text{for }  n \in \N 
\end{equation}
where $d \in \N^*$ and $p$ is a prime. Despite the presence of the infinity symbol in the above definition, we can show that for any odd prime $p$, the $p$-adic local density is in fact a finite sum.

\begin{prop}\label{prop formula for A_d(p^h,n)}
    Let $d \geq 1$ and $n \in \N$ and $p$ be an odd prime. If $d$ is even, then
    \begin{align}\label{ad even}
        A_d(p^h,n) = \begin{cases}
            \frac{p-1}{p} \rbra{\varepsilon_p^d p^{1-d/2}}^h&\quad\text{if }1 \leq h \leq \ord_p(n)\\
            -\rbra{\varepsilon_p^d p^{1-d/2}}^{\ord_p(n)+1}p^{-1} &\quad\text{if }h = \ord_p(n)+1\\
             0 &\quad\text{if }h > \ord_p(n)+1
        \end{cases}.
    \end{align}
    If $d$ is odd, then
    \begin{align}\label{ad odd}
        A_d(p^h,n) = \begin{cases}
            0&\quad\text{if }1 \leq h \leq \ord_p(n) \text{ and }h \text{ is odd}\\
            \frac{p-1}{p}  p^{(1-d/2)h}&\quad\text{if }1 \leq h \leq \ord_p(n) \text{ and }h \text{ is even}\\
           p^{(1-d/2)\ord_p(n)+(1-d)/2}\varepsilon_p^{d+1}\rbra{\frac{-n/p^{\ord_p(n)}}{p}} &\quad\text{if }h = \ord_p(n)+1\text{ and }h \text{ is odd}\\
            - p^{(1-d/2)\ord_p(n)-d/2} &\quad\text{if }h = \ord_p(n)+1\text{ and }h \text{ is even}\\
             0 &\quad\text{if }h > \ord_p(n)+1
        \end{cases}.
    \end{align}
Here, the symbols $\rbra{\frac{n}{p}}$ and $\varepsilon_p$ are defined in Definition \ref{defi extended symbol and varepsilon}.
\end{prop}
 
Direct applications of this result include Corollaries \ref{cor explicit formula for p-adic density}, \ref{cor singular series} and \ref{thm bound r(n)}. To begin, for any odd prime $p$ we are able to write the explicit formula for the $p$-adic local density.
\begin{cor}\label{cor explicit formula for p-adic density}
    Let $d \geq 3$ and $n \in \N$ and $p$ be an odd prime. If $d$ is even, then
    \begin{align}\label{explicit formula for p-adic d even}
    \delta_{p,d}(n) = C_{p,d} \rbra{1-\rbra{\varepsilon_p^d p^{1-d/2}}^{\ord_p(n)+1}},
    \end{align}
    where $C_{p,d} = (1-\varepsilon_p^d p^{-d/2})(1-\varepsilon_p^d p^{1-d/2})^{-1}$ is a positive constant which depends on $p$ and $d$.
    If $d$ is odd, then
    \begin{align}\label{explicit formula for p-adic d odd}
       \delta_{p,d}(n) = \begin{cases}
          p^{(1-d/2)\ord_p(n)} E_{p,d} + F_{p,d}&\quad \text{ if }\ord_p(n) \text{ is odd}\\
           p^{(1-d/2)\ord_p(n)}G_{p,d,n}+F_{p,d}&\quad \text{ if }\ord_p(n) \text{ is even}\\
       \end{cases},
    \end{align}
    where $E_{p,d} = p^{1-d/2}(p-1)(1-p^{2-d})^{-1}$ and $F_{p,d} = (1-p^{1-d})(1-p^{2-d})^{-1}$ are two positive constants which depend on $p$ and $d$; and $G_{p,d,n} = p^{1-d}(1-p)(1-p^{2-d})^{-1} + p^{(1-d)/2} \varepsilon_p^{d+1} \rbra{\frac{-n/p^{\ord_p(n)}}{p}}$ is a term which depends on $p,d,n$.
\end{cor}

One can derive, using for instance Hensel's lemma, that for $d \geq 5$, the 2-adic local density $\delta_{2,d}(n)$ is bounded from below by a positive constant.
\begin{lemma}[cf. {\cite[\S 8]{circle}}]
    For $d \geq 5$, there exists $C_0>0$ such that for any $n \in \N$ we have
    \begin{align*}
        \delta_{2,d}(n) \geq C_0.
    \end{align*}
\end{lemma}
Besides, by the defining form of $A_d(q,n)$, one can show that this is a multiplicative function of $q$ (cf. \cite[Lemma 7.4]{circle}), so that we can rewrite the singular series above in the form of a product of the $p$-adic local densities.
\begin{equation}\label{singular series product form}
    \mathfrak{S}_d(n) = \prod_p \delta_{p,d}(n).
\end{equation}
Moreover, Corollary \ref{cor explicit formula for p-adic density} implies that when $d \geq 5$, the singular series can be bounded from below by a positive constant which is uniform in $n$.

\begin{cor}\label{cor singular series}
   For $d \geq 5$, there exists $C>0$ such that for any $n \in \N$ we have
        \begin{align*}
           \mathfrak{S}_d(n) \geq C.
        \end{align*}
\end{cor}

If we denote the leading factor of $r_d(n)$ by
    \begin{align}\label{infity density}
        \delta_{\infty,d}(n) := \frac{\pi^{d/2}}{\Gamma(d/2)}n^{d/2-1},
    \end{align}
    then one can verify that it embodies the density of real solutions to $Q(x)=n$ (cf. \cite[Chapter 20]{MR2061214}). Hence, for $d\geq 5$ the representation number by $Q$ reads
    \begin{align*}
        r_d(n) = \delta_{\infty,d}(n) \prod_p \delta_{p,d}(n) + o(n^{d/2-1}).
    \end{align*}

Hence, by Theorem \ref{thm representation number} we have the following bound for the representation number by $Q$.
\begin{cor}\label{thm bound r(n)}
Let $d \geq 5$. The representation number by $Q$ satisfies
\begin{align}
    r_d(n) &\asymp n^{d/2 -1}.
\end{align}
\end{cor}

Let us state another property of $\delta_{p,d}(n)$ as follows.

\begin{lemma}\label{lemma p-adic density}
Let $d \in \N^*$, $p$ be a prime and $u \in \N^*$ such that $(p,u)=1$. Then
\begin{align}
    \delta_{p,d}(n) = \delta_{p,d}(u^2 n).
\end{align}
\end{lemma}
\begin{proof}
For any prime $p$ and $u \in \N^*$ such that $(p,u)=1$, we know by definition (\ref{gauss sum formula}) that for any $1\leq a \leq p$ we have $S(p,u^2a) = S(p,a)$. Hence, for $n \in \N^*$ it follows that $A_d(p,u^2n)= A_d(p,n)$. In general, for $h \geq 0$, using the fact that $A_d(q,n)$ is a multiplicative function in $q$, we know that $A_d(p^h,u^2n) = A_d(p^h,n)$. Hence, the conclusion follows by (\ref{defi p-adic local density}).
\end{proof}
\begin{rem}
The detailed calculations supporting Proposition \ref{prop formula for A_d(p^h,n)}
and Corollaries \ref{cor explicit formula for p-adic density} and
\ref{cor singular series} are available in version~1 of the extended arXiv
preprint of this paper \href{https://arxiv.org/abs/2503.03873v1}{arXiv:2503.03873v1}.
\end{rem}

\medskip
\subsection{Proof of Theorem \ref{theorem equidistribution}}
Let $d \in \N^*$, $d \geq 4$ and $p$ be an odd prime. Let $a \in \Z/p\Z$. We denote by $\pazocal{F}_p$ the set of functions $f: (\Z/p\Z)^d \to \C$ such that (\ref{vanishing condition}) holds.

(1) Suppose that $a \neq 0$. Let $\pazocal{F}_{p}(a) \subset \pazocal{F}_p$ be the subset of the functions on $(\Z/p\Z)^d$ such that $\sum_{x \in X_{p,d}(a)} f(x)=0$ and $f(x)=0$ on $X_{p,d}(a)^c$. It suffices to take $f \in \pazocal{F}_p(a)$ to test the equidistribution. In this case, whenever $n \equiv a [p]$ we know that $\sum_{x \in X_d(n)} f(x[p])$ is the $n$-th Fourier coefficient of the weighted theta function $\theta_f$ defined by (\ref{theta fourier series}). By Theorem \ref{cusp form}(2), we know that $\theta_f$ is a cusp form of weight $\frac{d}{2}$. Hence, for $d \geq 5$, the Hecke bound and Corollary \ref{thm bound r(n)} imply that over $n \equiv a [p]$ we have $\sum_{x \in X_d(n)} f(x[p])\frac{1}{r_d(n)} \ll \frac{n^{d/4}}{n^{d/2-1}}= n^{1-d/4}$. For $d=4$, Theorem \ref{thm bound kloosterman sums}, \ref{thm rn for d =3 or 4} and Remark \ref{rem r4} yield for any $\varepsilon >0$, over odd integers $n \equiv a [p]$, we have $\sum_{x \in X_4(n)} f(x[p])\frac{1}{r_4(n)} \ll \frac{n^{3/4+\varepsilon}}{n}= n^{\varepsilon-1/4}$. Therefore, for $d \geq 5$, we know $\sum_{x \in X_d(n)} f(x[p]) = o(r_d(n))$ over $n \equiv a[p]$. For $d=4$ and $\varepsilon$ small enough, we have $\sum_{x \in X_4(n)} f(x[p]) = o(r_4(n))$ over odd integers $n \equiv a[p]$. Meanwhile, for $f \in \pazocal{F}_p(a)$ we know that the right hand side of (\ref{eqn1 equi}) is $\frac{1}{\abs{X_{p,d}(a)}} \sum_{x \in X_{p,d}(a)} f(x) = 0$,
so the statement of Theorem \ref{theorem equidistribution}(1) holds.

(2) Suppose that $a=0$. We now state two auxiliary results used in the proof of the second part.
\begin{prop}\label{prop r_4 difference}
    Let $p$ be an odd prime. Then there exist $C>0$ such that for any odd $n \in \N$,
    \begin{align}
        r_4(p^2 n) - r_4(n) \geq C n.
    \end{align}
\end{prop}
\begin{proof}
    Let $n = p^k n_1$ for some $k \in \N$ such that $(n_1,p)=1$. Then by Theorem \ref{thm rn for d =3 or 4} it follows that
    \begin{align*}
     r_4(p^2n) = 8\sum_{d|p^{k+2}n_1,2\nmid d}d= 8\sum_{t=0}^{k+2}\sum_{d|n_1,2\nmid d}p^k d, \quad
     r_4(n) = 8\sum_{d|p^{k}n_1,2\nmid d}d = 8\sum_{t=0}^{k}\sum_{d|n_1,2\nmid d}p^k d.
    \end{align*}
    Hence,
    \begin{align*}
    r_4(p^2n)-r_4(n) =8\rbra{\sum_{d|n_1,2\nmid d}d} \rbra{p^{k+1}+p^{k+2}} \geq 8 \frac{n}{p^k}(p^{k+1}+p^{k+2}) = 8 (p+p^2) n.
    \end{align*}
Thus, the conclusion holds for $C := 8 (p+p^2)$.
\end{proof}
\begin{prop}\label{prop difference bound}
    Let $d \geq 5$ and $p$ be an odd prime. Then there exist $C,n_0 >0$ such that for $n \geq n_0$,
    \begin{align}
        r_d(p^2n)-r_d(n) \geq C n^{d/2-1}.
    \end{align}
\end{prop}
The following lemma, a consequence of Corollary \ref{cor explicit formula for p-adic density}, is central to proving the above proposition.
\begin{lemma}\label{lemma 3.16}
Let $d \geq 3$ and $p$ be an odd prime. Then there exists $C_0>0$ such that for any $n \in \N$ we have
\begin{align}
    p^{d-2} \delta_{p,d}(p^2n) - \delta_{p,d}(n) \geq C_0.
\end{align}
\end{lemma}
\begin{proof}
If $d$ is even, recall that $C_{p,d}$ in (\ref{explicit formula for p-adic d even}) from Corollary \ref{cor explicit formula for p-adic density} is a positive constant which depends on $p$ and $d$, then it follows that
\begin{align*}
    p^{d-2} \delta_{p,d}(p^2n) - \delta_{p,d}(n) &= C_{p,d} \bbra{ p^{d-2}\rbra{1-\rbra{\varepsilon_p^d p^{1-d/2}}^{\ord_p(n)+3}}-\rbra{1-\rbra{\varepsilon_p^d p^{1-d/2}}^{\ord_p(n)+1}}} \\
    &= C_{p,d} \bbra{p^{d-2}-1 + \varepsilon_p^{d(\ord_p(n)+1)}p^{(1-d/2)(\ord_p(n)+1)}(1-\varepsilon_p^{2d})}\\
    &= C_{p,d} \rbra{p^{d-2}-1} >0.
\end{align*}
If $d$ and $\ord_p(n)$ are both odd, then $\ord_p(p^2n) = \ord_p(n)+2$ is also odd. Recall $E_{p,d}$ and $F_{p,d}$ in (\ref{explicit formula for p-adic d odd}) from Corollary \ref{cor explicit formula for p-adic density} are two positive constants which depend on $p$ and $d$. It follows that
\begin{align*}
    p^{d-2} \delta_{p,d}(p^2n) - \delta_{p,d}(n) &= p^{d-2} p^{(1-d/2)(\ord_p(n)+2)} E_{p,d} + p^{d-2} F_{p,d} - p^{(1-d/2)\ord_p(n)} E_{p,d} -  F_{p,d}\\
    &= p^{(1-d/2)\ord_p(n)} \rbra{p^{d-2} p^{2-d} - 1} E_{p,d} + \rbra{p^{d-2}-1} F_{p,d} \\
    &= \rbra{p^{d-2}-1} F_{p,d} >0.
\end{align*}
If $d$ is odd and $\ord_p(n)$ is even, then $\ord_p(p^2n) = \ord_p(n)+2$ is also even. Besides,
\begin{align*}
    \rbra{\frac{-p^2n/p^{\ord_p(p^2n)}}{p}} = \rbra{\frac{-n/p^{\ord_p(n)}}{p}}.
\end{align*}
Hence, by definition of $G_{p,d,n}$ in (\ref{explicit formula for p-adic d odd}) we know that $G_{p,d,n} = G_{p,d,p^2n}$. It follows that
\begin{align*}
    p^{d-2} \delta_{p,d}(p^2n) - \delta_{p,d}(n) &= p^{d-2} p^{(1-d/2)(\ord_p(n)+2)} G_{p,d,p^2n} + p^{d-2} F_{p,d} - p^{(1-d/2)\ord_p(n)} G_{p,d,n} - F_{p,d}\\
    & = p^{(1-d/2)\ord_p(n)} \rbra{ p^{d-2} p^{2-d} - 1} G_{p,d,n}  + (p^{d-2}-1) F_{p,d}\\
    & = \rbra{p^{d-2}-1} F_{p,d} >0.
\end{align*}
In summary, if $d$ is even, one can take $C_0 := (p^{d-2}-1)C_{p,d}$; if $d$ is odd, one can take $C_0 := (p^{d-2}-1)F_{p,d}$.
\end{proof}
\begin{proof}[Proof of Proposition \ref{prop difference bound}]
For $d \geq 5$, by Theorem \ref{thm representation number} and using the notation (\ref{infity density}) we get
\begin{align*}
    r_d(p^2n) - r_d(n) &= \delta_{\infty,d}(p^2n)\prod_l\delta_{l,d}(p^2n) - \delta_{\infty,d}(n)\prod_l\delta_{l,d}(n) + o\rbra{n^{d/2-1}}\\
    &= \frac{\pi^{d/2}}{\Gamma(d/2)} n^{d/2-1}\rbra{p^{d-2}\prod_{l} \delta_{l,d}(p^2n) - \prod_{l} \delta_{l,d}(n)} + o\rbra{n^{d/2-1}}.
\end{align*}
Let us denote the term in the brackets by $D_{p,d}(n)$, then by Lemma \ref{lemma p-adic density} we know that
\begin{align*}
    D_{p,d}(n) &= p^{d-2}\prod_{l} \delta_{l,d}(p^2n) - \prod_{l} \delta_{l,d}(n) = \prod_{l \neq p} \delta_{l,d}(n) \rbra{p^{d-2}\delta_{p,d}(p^2n) - \delta_{p,d}(n)}.
\end{align*}
Now Corollary \ref{cor singular series} implies that there exists $C_1>0$ such that for $n \geq n_0$ we have $\prod_{l \neq p} \delta_{l,d}(n) > C_1$, which is uniform in $n$. Hence, we finally know that by Lemma \ref{lemma 3.16} for $n \geq n_0$ we have $D_{p,d}(n) \geq C_0C_1$ and the statement follows by taking $C := \frac{\pi^{d/2}}{\Gamma(d/2)} C_0C_1.$
\end{proof}
We now return to the second part of the proof of Theorem \ref{theorem equidistribution}.

Let $\tilde{\pazocal{F}}_p(0) \subset \pazocal{F}_p$ be the subset of the functions on $(\Z/p\Z)^d$ such that $\sum_{x \in X_{p,d}(0) \setminus \{0,\cdots,0\}}f(x)=0$ and $f(x)=0$ on $\rbra{X_{p,d}(0) \setminus \{0,\cdots,0\}}^c$. Then it suffices to take $f \in \tilde{\pazocal{F}}_p(0)$ to test the equidistribution. We note that
\begin{align*}
   \abs{\bbra{x \in \Z^d \setminus (p\Z)^d: Q(x,x)=n}} &= \abs{\bbra{x \in \Z^d : Q(x,x)=n}} - \abs{\bbra{x \in \Z^d: Q(x,x)=n/p^2}}\\
    &= \begin{cases}
    r_d(n),\quad &\text{if }n \in  p\Z \setminus p^2 \Z\\
    r_d(n) - r_d(n/p^2),\quad &\text{if } n \in p^2 \Z
    \end{cases}.
\end{align*}
Hence, for $n \in p\Z \setminus p^2 \Z$ it follows that $\sum_{\substack{ x \in \Z^d \setminus (p\Z)^d,\\ Q(x,x)=n}}f(x[p]) = \sum_{x \in X_d(n)}f(x[p])$, and all the arguments follow exactly as those from (1). For $n \in p^2\Z$, the representation number $r_d(n/p^2)$ is nontrivial and $\sum_{\substack{ x \in \Z^d \setminus (p\Z)^d,\\ Q(x,x)=n}}f(x[p]) =  \sum_{x \in X_d(n)}f(x[p]) - f(0,\cdots,0) r_d(n/p^2) = \sum_{x \in X_d(n)}f(x[p])$ again since $f \in \tilde{\pazocal{F}}_p(0)$. Now by Theorem \ref{thm bound kloosterman sums}, the Hecke bound and Proposition \ref{prop r_4 difference}, \ref{prop difference bound}, we can quickly conclude that the statement of Theorem \ref{theorem equidistribution}(2) holds.
\section{Remarks on equidistribution results in dimension 3}\label{section 4}
For low dimensions, one notices already in Theorem \ref{theorem equidistribution} that when $d=4$, the equidistribution fails over $n = 2^k, k \geq 1$.
For $d=3$, the formulas for the representation numbers are more subtle. By Siegel's theorem, we know that
 \begin{align}
      n^{1/2-\varepsilon} \ll r_3(n) \ll n^{1/2+\varepsilon}
    \end{align}
    for any $\varepsilon>0$, cf. \cite[\S 11.3]{MR1474964}.

Furthermore, as noted in Remark \ref{rem kloosterman sums}, for $k = 3/2$, the $n$-th Fourier coefficients satisfy the bound $a_n \ll n^{k/2 -1/4 + \varepsilon}$, based on a bound for Salié sums. In particular, if $n$ is square-free, \cite[Theorem 5.6]{MR1474964} provides a way to reduce the exponent $k/2-1/4$. However, this bound applies to cusp forms corresponding to the congruence subgroup $\Gamma_0(4)$, with weights $k \geq 5/2$.

Following Iwaniec's previous result, Duke and Schulze-Pillot (cf. \cite[Lemma 2]{MR1029390}) imposed another Shimura lifting-type condition to derive a useful bound for cusp forms under $\Gamma_1(N)$ of weights $k \geq 3/2$, which is relevant for our situation concerning non-zero level sets. However, further investigation of the zero level set requires a more refined understanding of the representation number $r_3(n)$.
 
\medskip
\printbibliography

\end{document}